\documentclass{siamltex}
\usepackage{amsmath}
\usepackage[dvips]{graphics}
\usepackage{amsfonts}
\usepackage{latexsym}
\usepackage{amsmath}
\usepackage{amssymb}
\usepackage{amscd}
\usepackage{color}
\usepackage{verbatim}
\usepackage[dvips]{graphics}
\usepackage[dvips]{graphicx}
\usepackage[dvips]{epsfig}
\newtheorem{thm}{Theorem}
\newtheorem{example}{\it Example}
\newtheorem{efinition}{efinition}[section]
\newtheorem{remark}[efinition]{Remark}

\title{On polynomial solvability of the Hamiltonian cycle problem
for graphs of degree less than or equal to 3}

\author{Ivan I. Goray\thanks{I.I.G., Inc., Discrete Mathematics
Division, 8 Obruchevyh str. Ap. 190, Saint Petersburg, 194064,
Russian Federation ({\tt lig@pcgrate.com}).}}

\begin{document}
\maketitle
\begin{abstract}
Any graph can be represented pictorially as a figure. Moreover, it
can be represented as two or more figures that can be have
different properties to each other. For the purpose of HCP, we
represent a graph by two such figures. In each of them, there is
an exterior part called the {\em contour}, and an interior part.
These two figures differ from each other by the constitution of
the edges in the interior part. That is, any edges in the interior
part for one figure are not in the interior for the other figure.
We call these two figures {\em basic objects}.

We develop rules and algorithms that allow us to represent any
graph of degree $d \leq 3$ by two basic objects. Individually,
neither of these representations possess the features to easily
determine the Hamiltonicity of the graph. However, the combination
of these two figures, once certain weights are assigned to their
edges, allows us to determine the Hamiltonicity with a
polynomial-time check.

The rules for the assignment of weights are:
\begin{enumerate}
\item The weight of any edge of the interior part is 0, for both
objects. \item In both figures any common edge of the contour has
the same weight.
\end{enumerate}

The weights of the edges allow us to extend the number of
parameters of the objects, that is sufficient to determine the
Hamiltonicity of the graph. Then, if the graph is Hamiltonian,
then both figures possess the same set of parameters. If the sets
of parameters for two figures are different, then the graph is not
Hamiltonian. The parameters that determine the Hamiltonicity of
the graph are the sums of weights of edges and {\em windows} of
contours in the figure. The algorithms of their construction do
not contain a combinatorial number of elements and have polynomial
complexity. We also supply an estimate of the complexity of each
algorithm.
\end{abstract}



\pagestyle{myheadings} \thispagestyle{plain} \markboth{I.
GORAY}{ON POLYNOMIAL SOLVABILITY...}

\section{Introduction}\label{sec1}
To date, a solution of the general problem P $\neq$ NP has not be
found. If this problem can be solved then it will be most likely
on the basis of some common features of the classes P and NP.
Looking for a solution of any particular P vs NP problem seems to
lack perspective. If it were possible to obtain, for a particular
NP problem, that we cannot solve it in polynomial time, it would
be necessary to supply the complete set of algorithms used in the
solution of this problem, which is most likely impossible.
Thinking deeper about the NP phenomenon and about the possibility
of solving such problems leads to some questions: Which
algorithms, in principle, could be constructed to solve particular
NP problems ? Can we avoid combinatorial enumeration in such
algorithms ? Can we overcome NP-completeness with the help of such
algorithms ? Garey and Johnson \cite{1} say that polynomial
algorithms can be constructed only when it is possible to
penetrate deep into the essence of the problem in question. The
success in solving particular NP problems using optimization
methods is very high (e.g. \cite{2,3,4,5}). Most likely, here is
the key of the difficulty of finding the precise solution, because
one can not substitute the essence of the problem by advances in
modern optimization technologies. It is commonly accepted
\cite{6,7,8} that the branch and bound method in all NP problems
is the most universal and effective. It is possible that branch
and bound methods mostly captures the difficulty of the problem as
it is based in investigation and transformation of the network.

The solution of the HCP that we offer here is not based on
optimization. It does not contain a progression to a better
solution. This solution is based on the analysis of structural
differences between two objects constructed on the basis of the
same non-oriented graph with degree less than equal to 3.

\section{Objects}\label{sec2}
\subsection{Basic object}
\subsubsection{Properties of basic object}

\begin{figure}
\centering
\includegraphics[width=0.3\textwidth]{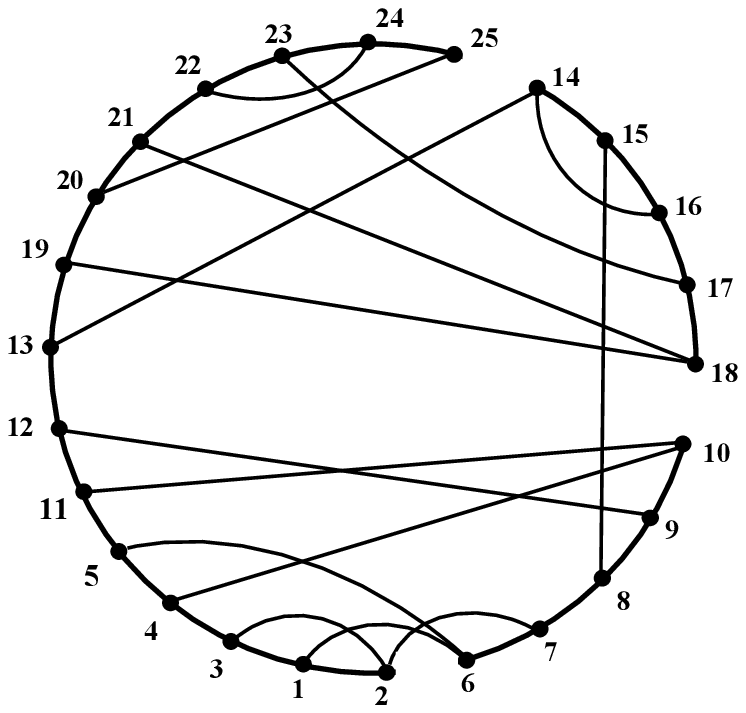} 
\caption{An example of a first basic object and its properties.
Nodes forming windows are 2, 6, 10, 18, 14, 25. Windows are
$w(2-6)$, $w(10-18)$, $w(14-25)$. Free edges are $e(22-24)$,
$e(9-12)$, $e(8-15)$, $e(17-23)$. Segments are $S_{2,25}$,
$S_{14,18}$, $S_{6,10}$.\label{fig2-1}}
\end{figure}
Given a graph $\Gamma$ of size $N$, we define object $G_N$ as a
figure representing $\Gamma$ that, by some property or properties,
differs from other figures representing the same graph $\Gamma$.
We can solve the HCP on a set of objects, and construct a
parameter for each object such that the parameters should not
coincide unless the graph is Hamiltonian. While, individually,
these objects do not solve HCP, the solution of the problem is
obtained on a finite set of the objects, by considering their
special properties. In fact, for any two such figures, we will
find a parameter that will distinguish between them.

For certain special cases, it is only possible to construct a
single, unique figure. In this case, the solution of the HCP is
given trivially by construction of this figure. However, in
general, for the solution of HCP, one has to construct two objects
that differ in the constitution of edges in different parts of the
objects, and in the method of construction of those parts. As an
example, in Figure 2.1, we display the first (basic) object
constructed from a 25-node graph.
\begin{example}
For convenience, we will distinguish between the exterior part of
the figure, that we call the contour, and the interior part. The
contour contains the edges of the graph and so-called {\em
windows}. A window is an edge that does not belong to the graph.
In the example in Figure 2.1, there are three windows: $w(2-6)$,
$w(10-18)$, and $w(14-25)$, that are formed by the nodes 2 and 6,
10 and 15, and 14 and 25 of the graph $\Gamma$.

The edges of the interior part of the object, that are not
connected at either end to the nodes that form the windows (for
convenience, we refer to these nodes as {\em window nodes}), are
called {\em free} edges. In the example given in Figure 2.1, the
free edges are $e(22-24)$, $e(17-23)$, $e(8-15)$, $e(9-12)$.

As follows from the example, the basic object has the following
properties.
\begin{enumerate}
\item Window nodes are not directly connected by the edges of the
interior part. We refer to such edges as {\em links}. \item The
basic object contains a {\em maximal} number of windows (three in
this example). That is, no further windows can be inserted without
violating the previous property. \end{enumerate}

The contour is partitioned by windows into several parts, and can
be determined and enumerated. The parts of the object separated by
windows will be called {\em segments}. In the example in Figure
2.1, there are three segments, between nodes 2 and 25, 14 and 18,
and 6 and 10, which we denote by $(S_{2,25}$, $S_{14,18}$,
$S_{6,10})$. A segment is called {\em degenerate} if it consists
of a single node such that has a window on both sides in the
contour.
\end{example}

The basic objects that are constructed for the HCP should satisfy
the following properties.
\begin{enumerate}
\item[(1)] The number of windows in the object should be maximal.
As we show later, we will use a simple algorithm to increase the
number of windows, and we will determine the criterion that
signals the end of the algorithm. \item[(2)] Window nodes should
not be connected to other window nodes by the edges of the
interior part. \item[(3)] The object should not contain degenerate
segments if the degree of the graph is $d = 3$. Degenerate
segments can only exist for nodes of degree $d = 2$. \item[(4)]
For a given graph, the set of edges in the interior part should be
different for any pair of basic objects.
\end{enumerate}

{\bf Lemma 1.} For a degenerate segment-free graph $\Gamma$ of
size $N$, and degree $d \leq 3$, the number of segments in all
basic objects is less than or equal to $\nu = [N/6]$.
\begin{proof}First, we observe that, by property 1, the nodes that
form the windows are not interconnected. Therefore, every node
that forms a window is connected by the edges of the interior part
with another two nodes that do not form windows. From the fragment
of the graph shown in Figure 2.2, we see that window $w(r-j)$,
together with nodes $a, b, c$ and $e$ constitutes a set of six
nodes in the graph. Since in the absence of degenerate segments
this is true for every window, then there can not be more than
$\nu = [N/6]$ segments.\end{proof}
\begin{figure}
\centering
\includegraphics[width=0.25\textwidth]{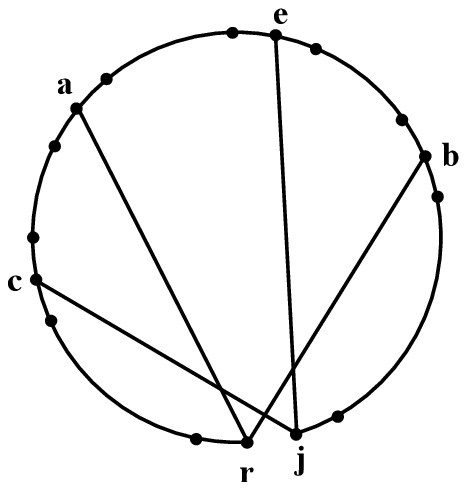}
\caption{The nodes $a$, $b$, $c$, $e$, $r$ and $j$ that correspond
to a window $w(r-j)$, and its incident nodes. None of these nodes
can be incident to any other window.\label{fig2-2}}
\end{figure}

\begin{remark}
\begin{enumerate}
\item The number of windows can be larger if there are degenerate
segments. \item If there are free edges, then the number of free
edges is less than $\nu$. \item If the degree of the graph was not
bounded by 3 then the number of windows in the basic object
constructed for the graph would be strictly less than $\nu$. This
is because the nodes that form the windows for $d > 3$ will be
connected to a larger number of nodes in the graph that, in view
of property 2, can not form windows
themselves.\end{enumerate}
\end{remark}

\subsubsection{Construction of the basic object}

\begin{figure}
\centering
\includegraphics[width=0.5\textwidth]{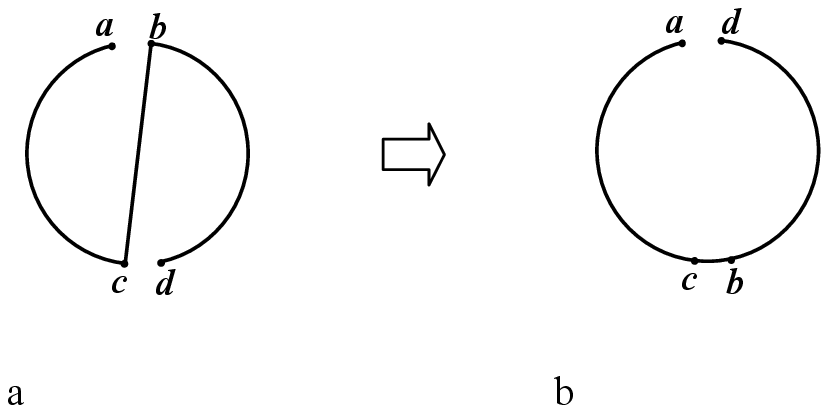}
\caption{A transformation that eliminates a link $e(b-c)$ between
windows $w(a-b)$ and $w(c-d)$ (in Figure 2.3a). Edge $e(b-c)$ is
moved to the contour, and windows $w(a-b)$ and $w(c-d)$  are
substituted by a single window, $w(a-d)$ (in Figure
2.3b).\label{fig3ab}}
\end{figure}

Consider the problem of constructing a basic object that satisfies
properties (1)--(4) above. This can be performed in two steps. In
the first step, we can use elementary procedures that are from
known optimization algorithms \cite{9}, that are described below.

{\em Step 1: Construction of a graph whose windows are not
connected by the edges of the interior part.}

It is known that optimization algorithms ensure an optimal
solution is found by gradually eliminating links between the
windows by means of elementary network transformations. For this
reason, they can be used to construct a basic object such that
nodes that form the windows are not interconnected. It is useful
to note that for the first step it is sufficient to use the
procedure of checking and substitution of no more than two edges
on the contour.
\begin{figure}
\centering
\includegraphics[width=0.5\textwidth]{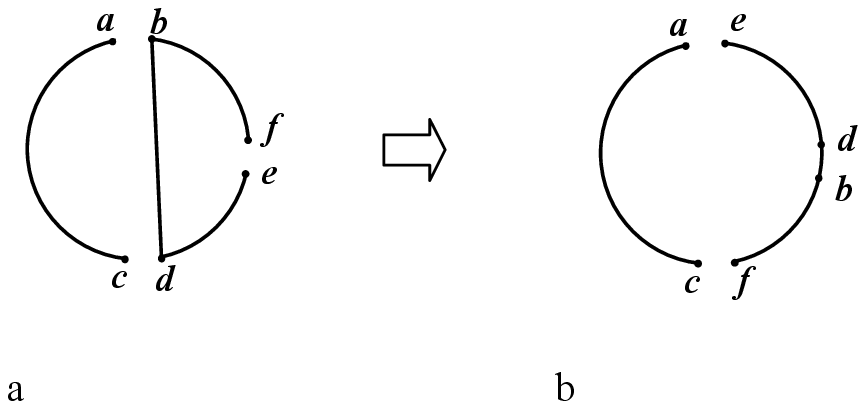}
\caption{A transformation that eliminates a link $e(b-d)$ between
windows $w(a-b)$ and $w(c-d)$ (in Figure 2.4a). Edge $e(b-d)$ is
moved to the contour, and windows $w(a-b)$, $w(c-d)$, $w(e-f)$ are
substituted by two windows, $w(a-e)$ and $w(c-f)$ (in Figure
2.4b).\label{fig4ab}}
\end{figure}

There are four possible cases of elimination of links:

{\bf Case 1.1}: We use the transformation that substitutes two
windows $w(a-b)$, $w(c-d)$ (see Figure 2.3a), with a single window
$w(a-d)$ and edge $e(b-c)$ (see Figure 2.3b). In this case, the
interior edge between nodes $b$ and $c$ is moved to the contour.

{\bf Case 1.2}: We use the transformation that substitutes three
windows $w(a-b)$, $w(c-d)$, $w(e-f)$ (see Figure 2.4a), with two
windows $w(a-e)$, $w(c-f)$, and the interior edge $e(b-d)$ is
moved to the contour (see Figure 2.4b).

{\bf Case 1.3}: Suppose the contour contains a segment $S_{b,f}$,
and its endpoints are connected with an interior edge $e(b-f)$.
Suppose that $S_{b,f}$ contains a node $l$ that is a neighbour to
node $m$ on the contour, and that is linked by an interior edge
$e(l-c)$ with another node that forms a window (see Figure 2.5a).
In this case, we use the transformation that substitutes three
windows $w(a-b)$, $w(c-d)$, $w(e-f)$ with two windows $w(a-e)$,
$w(m-d)$, and moves two interior edges $e(b-f)$ and $e(l-c)$ to
the contour (see Figure 2.5b). The edge $e(l-m)$ is moved to the
interior.
\begin{figure}
\centering
\includegraphics[width=0.5\textwidth]{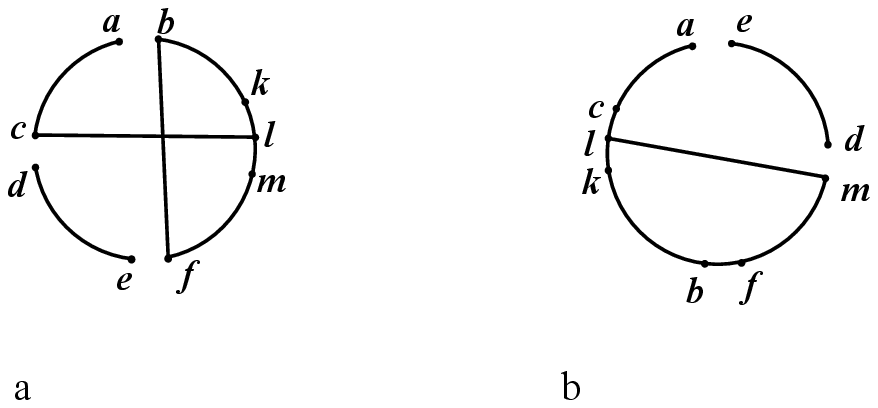}
\caption{A transformation that eliminates a link $e(b-f)$ between
windows $w(a-b)$ and $w(e-f)$ (in Figure 2.5a). Edges $e(b-f)$ and
$e(c-l)$ are moved to the contour, and windows $w(a-b)$, $w(c-d)$
and $w(e-f)$ are substituted by two windows, $w(a-e)$ and $w(d-m)$
(in Figure 2-5b).\label{fig2-5}}
\end{figure}

{\bf Case 1.4}: This case differs from case 3 in that $S_{b,f}$
does not contain any nodes that are connected with a node that
forms a window. In this case, we select a node $l$ on $S_{b,f}$
for which an interior arc $e(l-s)$ goes to a different segment
that contains a node $t$ that is a neighbour to node $s$ on the
contour (see Figure 2.6a). Then, we use the transformation that
moves edges $e(b-f)$, $e(s-l)$ to the contour, and substitutes
windows $w(a-b)$ and $w(e-f)$ with two windows $w(a-t)$, $w(e-n)$
(see Figure 2.6b).
\begin{figure}
\centering
\includegraphics[width=0.5\textwidth]{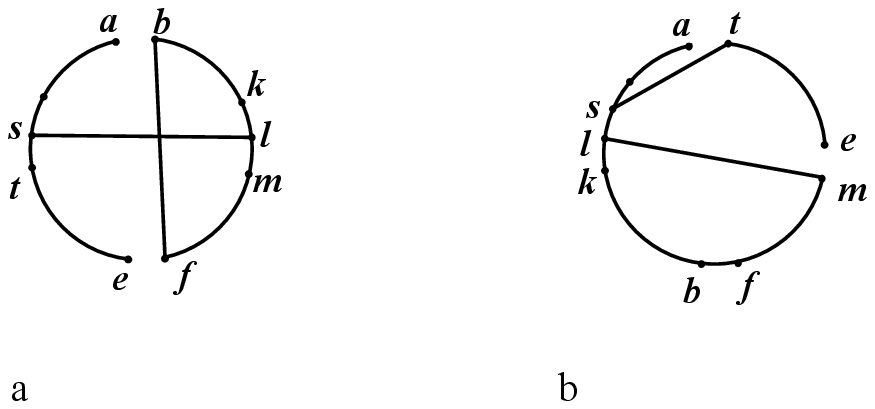}
\caption{A transformation that eliminates a link $e(b-f)$ between
windows $w(a-b)$ and $w(e-f)$ (in Figure 2.6a). Edges $e(b-f)$ and
$e(l-s)$ are moved to the contour, and windows $w(a-b)$ and
$w(e-f)$ are substituted by two windows, $w(a-t)$ and $w(e-m)$ (in
Figure 2.6b).\label{fig2-6}}
\end{figure}

In all of these cases, an interior edge that joins two window
nodes is eliminated from the set of interior edges. The object
that we construct after Step 1 can contain an arbitrary number of
windows, but cannot contain degenerate segments formed by nodes of
degree $d = 3$. We can achieve this by elimination of degenerate
segments from the object as described as follows.

{\bf Case 1.5}: Suppose that a degenerate segment contains node
$a$ that form two windows $w(a-b)$, $w(a-c)$ (see Figure 2.7a).
Suppose that node $a$ has degree $d = 3$, and is connected to
nodes $k$, $l$ and $n$, none of which form windows. To eliminate
$S_{a,a}$, we use the transformation that moves two interior edges
$e(a-l)$ and $e(a-m)$ to the contour, and substitutes windows
$w(a-b)$ and $w(a-c)$ with two windows $w(s-b)$, $w(t-c)$ (see
Figure 2.7b).
\begin{figure}
\centering
\includegraphics[width=0.5\textwidth]{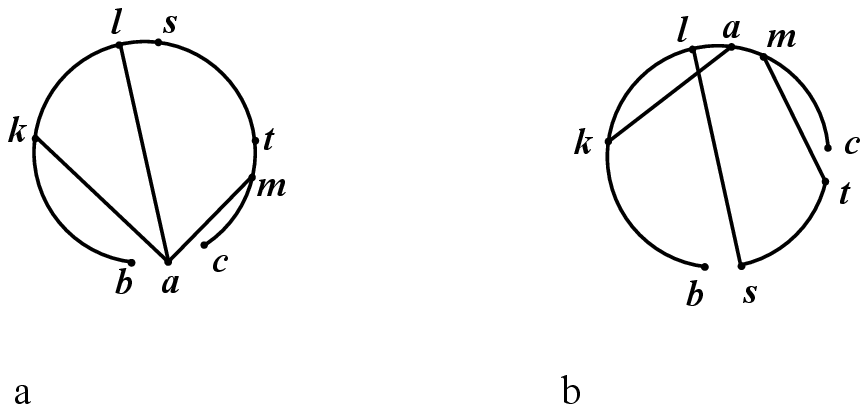}
\caption{A transformation that eliminates a degenerate segment
containing a node with degree $d = 3$ (in Figure 2.7a). Edges
$e(a-l)$ and $e(a-m)$ are moved to the contour, and windows
$w(a-b)$ and $w(a-c)$ are substituted by two windows, $w(b-s)$,
$w(c-t)$(in Figure 2.7b).\label{fig2-7}}
\end{figure}

{\bf Case 1.6}: If at least one of nodes $k$, $l$ and $m$ forms a
window, then the degenerate segment $S_{a,a}$ is eliminated from
the object by either of Cases 1.1 and 1.2 above.

We note that the above cases are sufficient to eliminate all
interior edges between windows, and all degenerate segments, which
completes Step 1.

{\em Step 2: Introduction of additional windows}

We can introduce additional windows if there are special
structures involving the free edges in the object formed in Step
1. Recall that the free edges are those that are not incident to
any window nodes. We note that introduction of additional windows
should not introduce any links, or else property 1 will be
violated.

The windows are added by means of elementary transformations of
one free edge to the contour, and two edges of the contour to the
interior. There are two cases when it is possible to introduce an
additional window to the object formed in Step 1:
\begin{figure}
\centering
\includegraphics[width=0.5\textwidth]{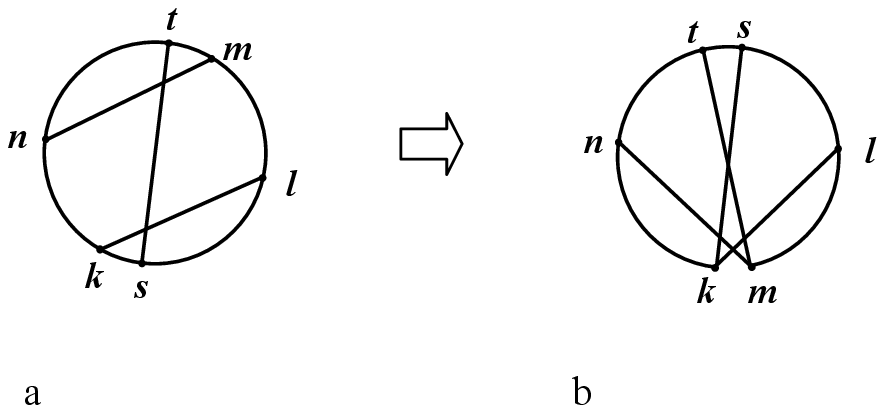}
\caption{A transformation that creates a windows $w(k-m)$ (in
Figure 2.8b), using free edges $e(s-t)$, $e(k-l)$ and $e(m-n)$ (in
Figure 2.8a).\label{fig2-8}}
\end{figure}

{\bf Case 2.1}: Suppose that an object formed in Step 1 contains
free edges $e(k-l)$, $e(m-n)$ and $e(s-t)$ such that $m$ is a
neighbour of $t$ on the contour, $k$ is a neighbour of $s$ on the
contour, and $k$ and $m$ are on different sides of edge $e(s-t)$.
Nodes $l$ and $n$ can be anywhere on the contour (see Figure
2.8a). An additional window $w(k-m)$ is introduced by the
transformation that moves interior edge $e(s-t)$ to the contour,
and moves edges $e(k-s)$ and $e(t-m)$ to the interior (see Figure
2.8b). This transformation is also valid if one or both of the
edges $e(k-l)$ and $e(m-n)$ are not present. In this case, one or
both of nodes $k$ and $m$ will have degree 2.

{\bf Case 2.2}: Suppose than on object formed in Step 1 contains
free edges $e(k-l)$, $e(m-n)$ and $e(s-t)$, such that $m$ is a
neighbour of $t$ on the contour, and $k$ is a neighbour of $s$ on
the contour, and nodes $k$, $m$ and $n$ are all on the same side
of edge $e(s-t)$, and a window $w(a-b)$ exists (see Figure 2.9a).
An additional window can be formed by a transformation of the part
of the contour bounded by nodes $m$ and $k$. In Figure 2.9a, this
is the part of the contour to the left of edge $e(s-t)$) to become
a segment between nodes $a$ and $b$ (see Figure 2.9b). As a result
of this transformation, we substitute window $w(a-b)$ with windows
$w(a-k)$ and $w(b-m)$. This transformation is also valid if one or
both of the edges $e(k-l)$ and $e(m-n)$ are not present. In this
case, one or both of nodes $k$ and $m$ will have degree 2.
\begin{figure}
\centering
\includegraphics[width=0.5\textwidth]{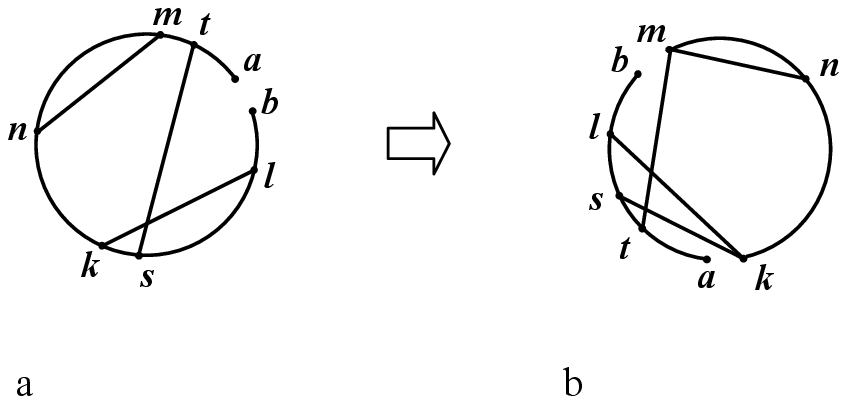}
\caption{A transformation that creates an additional window by
moving a part of the contour bounded by $m$ and $k$ between $a$
and $b$ that form window $w(a-b)$ (in Figure 2.9a). This
transformation substitutes window $w(a-b)$ by two windows,
$w(a-k)$ and $w(b-m)$ (in Figure 2.9b).\label{fig2-9}}
\end{figure}

Therefore, in both cases, if we have a special configuration of
free edges, we can construct additional windows so that the edges
$e(k-l)$, $e(m-n)$, $e(s-t)$ are no longer free. If the object
contains configurations analogous to those considered above (see
Figures 2.8a, 2.9a), then additional edges should be moved to its
contour. We note that none of these transformations should be
performed if they result in degenerate segments that contain a
node whose degree $d = 3$.

A basic object is completed if it is does not contain free edges
that could lead to the introduction of additional windows using
the transformations described above. In such a case, we say that
the number of windows in the object is maximal. Note that
depending on the construction of the graph, it may be possible to
construct a different object containing more windows, but it is
not necessary to do so. A basic object may contain one or more
free edges, if the configuration of those free edges does not
coincide with those described above. This is the only criteria
required to complete the second step of the construction of a
basic object.

\subsection{Second basic object}

The second basic object should possess the same properties (see
Subsection 2.1.1) as the first basic object. The second basic
object is constructed from the first basic object. The two basic
objects will have no common interior edges. This requires us to
introduce additional restrictions in the algorithm for
constructing the second basic object, that is otherwise analogous
to the algorithm that constructs the first basic object. These
restrictions concern the structure of the contour, and specify the
nodes which cannot not become window nodes in the second basic
object. We satisfy these new requirements in the first step of the
construction of the second object.

\subsubsection{Preliminary formation of the second basic object}

Suppose that the first basic object is constructed, and satisfies
all of the required properties. Then, we identify:
\begin{enumerate}
\item[(1)] The window nodes, and the interior edges incident to
these nodes. The number of such edges will be either two (if the
node has degree $d = 3$), or one (if the node has degree $d = 2$).
\item[(2)] The free edges. \item[(3)] Degenerate segments that
contain a single node of degree $d = 2$. \item[(4)] The set of
nodes with degree $d = 2$ that do not form windows.
\end{enumerate}

Note that the first basic object cannot contain any interior edges
or nodes of degree $d = 2$ that can be different from the cases
described above. This phenomenon is determined by the properties
of the construction of the first basic object: the original graph
has degree $d \leq 3$; the windows of the first basic object are
not linked by interior edges; and the first basic object does not
contain degenerate segments of degree $d = 3$.

The first basic object is constructed without restrictions on the
contour edges, interior edges, or on the configuration of nodes
with degree $d = 2$. The construction of the second basic object
depends on the particular structure of the first basic object.

{\em Algorithm of the preliminary construction of the second basic
object:}
\begin{enumerate}
\item[(1)] Identify in the first basic object:
\begin{enumerate}
\item {\em Chains} that consist of two interior edges that are
both incident to a window node of degree $d = 3$. \item Interior
edges that are incident to the window nodes of degree $d = 2$.
These nodes will form windows in the second basic object as well.
\item Free edges. \item Chains that consist of two interior edges
incident to a node of degree $d = 2$ that form a degenerate
segment. \item Nodes of degree $d = 2$ that do not form windows.
These nodes will form degenerate segments in the second basic
object.
\end{enumerate}
\item[(2)] The identified chains consisting of two edges, free
interior edges, and the degenerate segments of degree $d = 2$ all
move to the contour, and are all separated by windows. \item[(3)]
Connect the nodes of the contour in the second basic object by
interior edges that belong to the contour of the first basic
object. \item[(4)] Now, we eliminate some interior edges some
interior edges from the second basic object that cannot be moved
to the contour. These are edges that are:
\begin{enumerate}
\item Incident to window nodes of degree $d = 2$. \item Incident
to window nodes in the first basic object, and belong to its
contour. \end{enumerate} \item[(5)] Mark those nodes that cannot
form windows. By means of the first two algorithms of eliminating
links between windows (see Subsection 2.1.2, Cases 1.1 and 1.2)
move, to the contour, interior edges incident to these marked
nodes. \item[(6)] Mark edges of the contour that cannot be moved
to the interior.
\end{enumerate}
\begin{figure}
\centering
\includegraphics[width=0.5\textwidth]{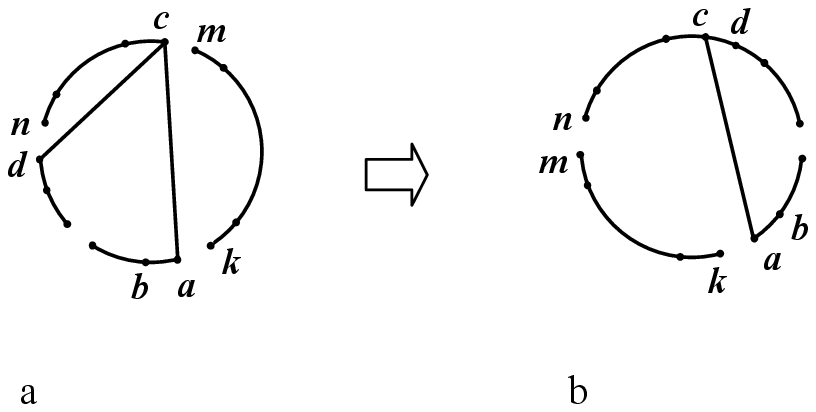}
\caption{A transformation that allows us to delete node $c$ from
the list of candidates for window nodes by elimination of the link
between windows $w(c-m)$, $w(d-n)$ (see Figures 2.10a and
2.10b).\label{fig2-10}}
\end{figure}
\begin{figure}
\centering
\includegraphics[width=0.8\textwidth]{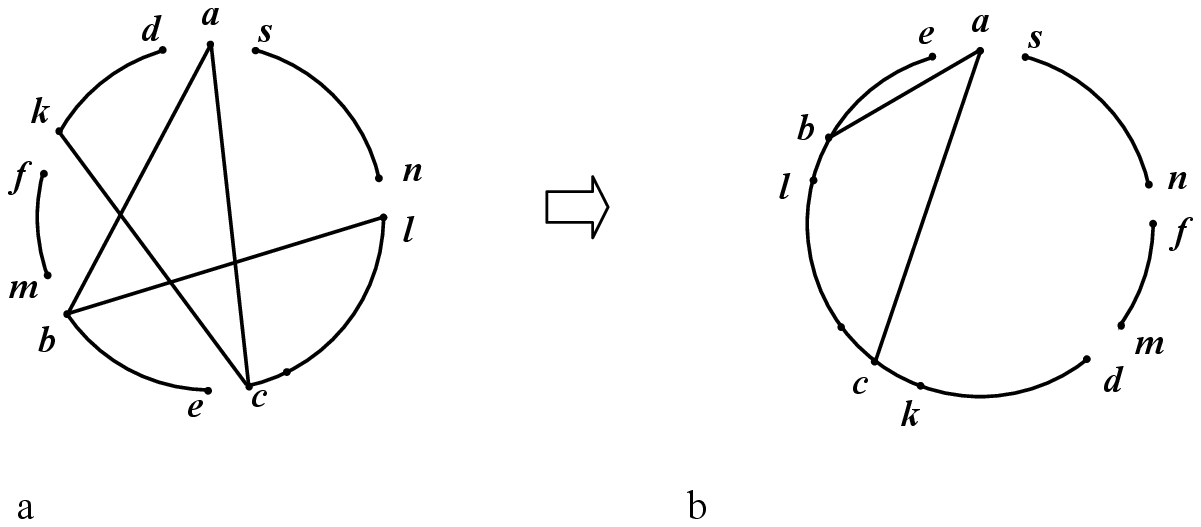}
\caption{A transformation that allows us to delete nodes $b$ and
$c$ from the list of candidates for window nodes by elimination of
the link between windows $w(b-m)$, $w(l-n)$ (edge $e(b-l)$), and
between windows $w(c-e)$, $w(k-f)$ (edge $e(k-c)$) (see Figures
2.11a and 2.11b).\label{fig2-11}}
\end{figure}

\begin{remark}
\begin{enumerate}\item[(A)]Consider, in more detail,
the possible situations of the appearance of nodes in the second
basic object that cannot form windows (corresponding to step (5)
above), and the elimination of those nodes. Suppose that in the
first basic object, there are window nodes of degree $d = 2$.
Then, if edge $e(a-c)$ that is incident to node $a$ belongs to the
contour of the first basic object, then the second basic object
will contain on the contour an edge $e(a-b)$ incident to the same
node $a$, which again forms a window in the second basic object
(see Figure 2.10a). In the second basic object the node $c$ should
not form a window. If it does, then the second basic object
violates the rules outlined in Section 2.  In this case, by
removing a link between windows $w(c-m)$ and $w(d-n)$, we can
delete node $c$ from the candidates to form windows in the second
basic object. We eliminate this link (interior edge $e(c-d)$) in
one of two ways described in Section 2. In Figure 2.10a, the
positions of windows $w(c-m)$ and $w(d-n)$ imply that the first
method should be used.  In another configuration of those windows,
the second method would be used. The situation when nodes $a$ and
$c$ have degree $d = 2$ (edge $e(c-d)$ is not present) cannot
happen, because two adjacent nodes of degree $d = 2$ are
substituted by a single node of degree $d = 2$. \item[(B)] Suppose
that the first basic object contains node $a$ of degree $d = 2$
that does not form a window, and edges $e(a-b)$ and $e(a-c)$
belong to the contour. Then, node $a$ should form a degenerate
segment in the second basic object, and nodes $b$ and $c$ should
not form windows (see Figure 2.11a). If this condition was
violated, then the second basic object would have windows linked
by interior edges, which would violate the rules outlined in
Subsection 2.1.1. As in item (A) above, nodes $b$ and $c$ are
eliminated from the set of nodes that can form windows in the
second basic object by temporarily deleting edges $e(b-l)$,
$e(c-k)$, between pairs of windows $w(b-m)$ and $w(l-n)$, and
$w(c-e)$ and $w(f-k)$ respectively (see Figure 2.11b). The
situation when nodes $a$ and $c$, and $a$ and $b$ have degree $d =
2$ (edges $e(b-l)$ and $e(c-k)$ are absent) is not possible, as
two or more adjacent nodes of degree $d = 2$ can be substituted by
a single node of degree $d = 2$.
\end{enumerate}
We note that elimination of links between windows by means of the
algorithms described in Subsection 2.1.2, Cases 1.1 and 1.2, are
performed so as to substitute one of the two windows of the
contour by a single interior edge. By doing this, neither edge of
the contour is moved to the interior.
\end{remark}
\begin{remark}
\begin{enumerate}
\item[(A)] One of the properties of the basic objects is that they
have a disjoint set of interior edges. In the preliminary
construction of the second basic object, the interior edges of the
first basic object are all moved to the contour of the second
basic object. During the subsequent steps, their movement to the
interior will be prohibited. \item[(B)] This restriction on moving
these particular contour edges to the interior is achieved by
temporarily deleting interior edges that cannot belong to the
contour (item (4) above). Then, the subproblems of the
construction of the second basic object whose windows are not
linked by interior edges, and of the introduction of additional
windows into the second basic object, are performed as outlined in
Subsection 2.1.2. This completes the preliminary construction of
the second basic object.\end{enumerate}\end{remark}

\subsubsection{Completion of the second basic object}

\begin{figure}
\centering
\includegraphics[width=0.4\textwidth]{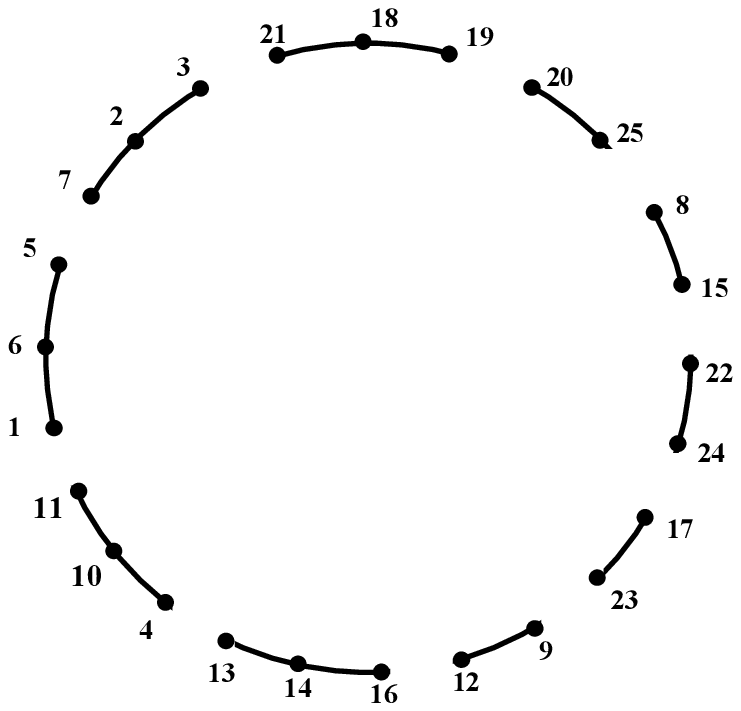}
\caption{Configuration of individual edges, chains of two edges,
and windows needed to form the contour of the second
object.\label{fig2-12}}
\end{figure}
\begin{figure}
\centering
\includegraphics[width=0.4\textwidth]{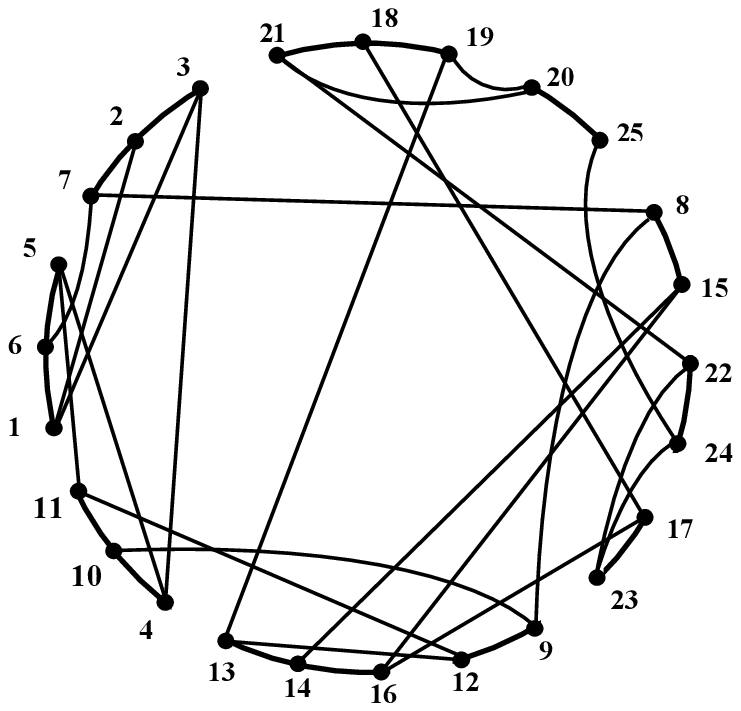}
\caption{The second basic object, in which the missing interior
edges belonging to the contour of the first basic object are
added.\label{fig2-13}}
\end{figure}
\begin{figure}
\centering
\includegraphics[width=0.4\textwidth]{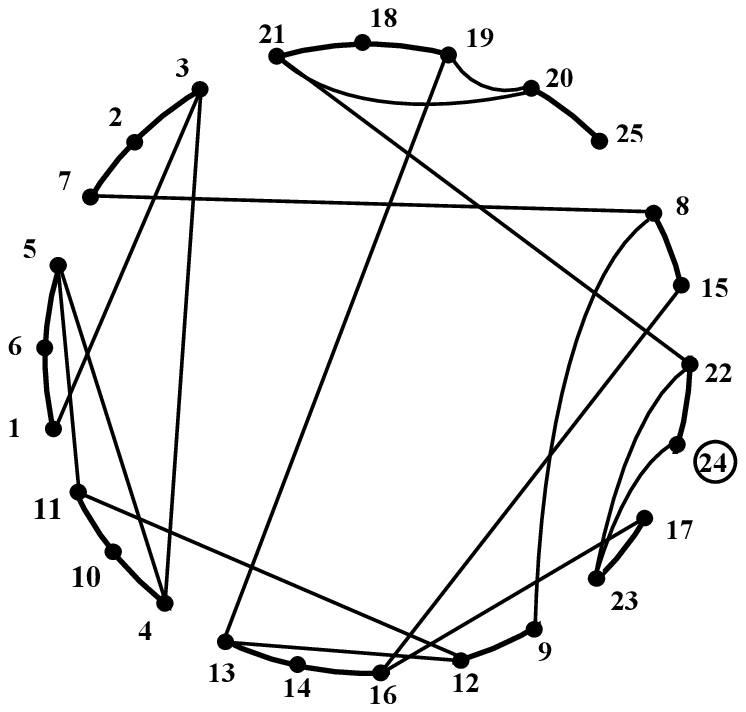}
\caption{The second basic object, with deleted interior edges that
can not belong to its contour. Node 24 that can not form a window
is marked.\label{fig2-14}}
\end{figure}

Construction of the second basic object is performed using the
same algorithm as the construction of the first basic object, that
is, the construction of an object whose windows are not linked by
interior edges, and the subsequent introduction of additional
windows.

Consider the following example.
\begin{example} In Figure 2.1, the first basic object for a certain
25-node graph is displayed. The construction of the second basic
object proceeds as follows:
\begin{enumerate}\item[(1)] In the first basic object, we
select:
\begin{enumerate}\item Chains of the interior that consist of two
edges incident to the window nodes: $e(16-14)$, $e(14-13)$;
$e(4-10)$, $e(10-11)$; $e(1-6)$, $e(6-5)$; $e(7-2)$, $e(2-3)$;
$e(21-18)$, $e(18-19)$. \item Interior edges that are incident to
window nodes of degree $d = 2$: $e(20-25)$. \item Free edges:
$e(8-15)$, $e(22-24)$, $e(17-23)$, $e(9-12)$.
\end{enumerate}
\item[(2)] The above selected edges separated by windows, form the
preliminary contour (see Figure 2.12): (16-14-13 W 4-10-11 W 1-6-5
W 7-2-3 W 21-18-19 W 20-25 W 8-15 W 22-24 W 17-23 W 9-12).
\item[(3)] Connect nodes (see Figure 2.13) by interior edges that
were not selected in (1), and correspond to contour edges in the
first basic object: $e(1-2)$, $e(1-3)$, $e(3-4)$, $e(4-5)$,
$e(5-11)$, $e(11-12)$, $e(12-13)$, $e(13-19)$, $e(19-20)$,
$e(20-21)$, $e(21-22)$, $e(22-23)$, $e(23-24)$, $e(24-25)$,
$e(14-15)$, $e(15-16)$, $e(16-17)$, $e(17-18)$, $e(9-10)$,
$e(8-9)$, $e(7-8)$, $e(6-7)$. \item[(4)] From the set of interior
edges, we temporarily remove those edges that cannot be moved to
the contour of the second basic object: $e(17-18)$, $e(14-15)$,
$e(6-7)$, $e(1-2)$, $e(9-10)$, $e(24-25)$ (see Figure 2.14). Once
these edges are temporarily removed, node 25 becomes an end node
(that is, node 25 acquires degree $d = 1$), and nodes 24, 9, 15,
1, 7, 17 acquire degree $d = 2$. \item[(5)] Mark the nodes that
cannot form windows. In this example, node $24$ can not form a
window. This node is circled (see Figure 2.14). By the algorithm
of eliminating links between window nodes (see Subsection 2.1.2,
Case 1.2), we move the edge $e(23-24)$, incident to node 24, to
the contour. This completes the preliminary construction of the
second basic object (see Figure 2.15).\end{enumerate}
\begin{figure}
\centering
\includegraphics[width=0.4\textwidth]{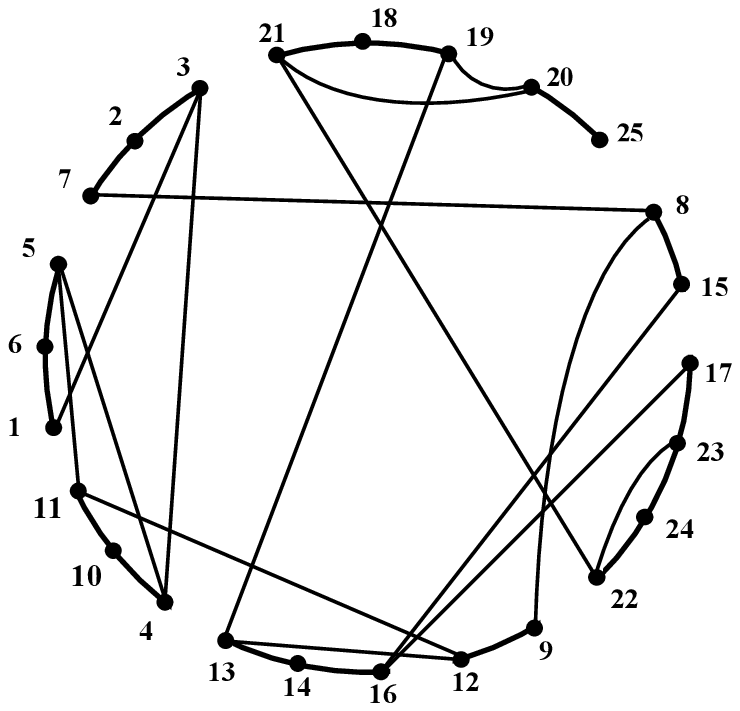}
\caption{An object in which node 24 is deleted from the list of
nodes that can form windows, by means of eliminating the link
between windows $w(9-23)$, $w(17-24)$.\label{fig2-15}}
\end{figure}
\begin{figure}
\centering
\includegraphics[width=0.4\textwidth]{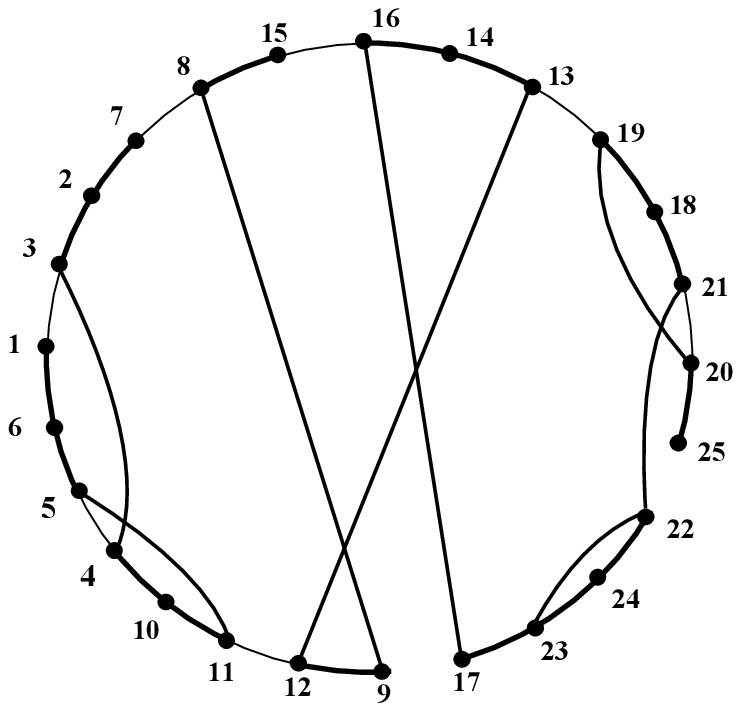}
\caption{A segment of the second object where the links between
windows are eliminated by means of interior edges. Free edges
$e(19-20)$, $e(5-11)$, $e(3-4)$, $e(12-13)$ are identified that
determine a possibility of forming additional
windows.\label{fig2-16}}
\end{figure}

We note that the second object at this stage does not contain
window nodes of degree $d = 2$. For this reason, this second
object does not contain degenerate segments. The edges of the
contour of the first basic object that should also belong to the
contour of the second basic object, and therefore cannot be moved
to its interior, are displayed in bold (see Figure 2.15).

An object whose windows are not linked by interior edges (see
Figure 2.16) is constructed using the algorithm described in
\cite{9} (which uses the steps outlined in Subsection 2.1.2, Cases
2.1--2.4). At this stage, construction of the object whose windows
are not linked by interior edges is complete.

The next step is to introduce additional windows. We use the same
algorithm (in Subsection 2.1.2) that was used in the construction
of the first basic object. If we encounter an instance where an
additional window can be introduced by means of free edges, then
we introduce this window and continue to search until all
instances are exhausted. In our example (see Figure 2.16), there
are at most four possibilities of introduction of an additional
windows. The object contains four free edges: $e(12-13)$,
$e(3-4)$, $e(9-11)$, $e(19-20)$, that determine the number of
possibilities. Consider all possible ways to introduce an
additional window by using any of the listed four free edges.
\begin{figure}
\centering
\includegraphics[width=0.4\textwidth]{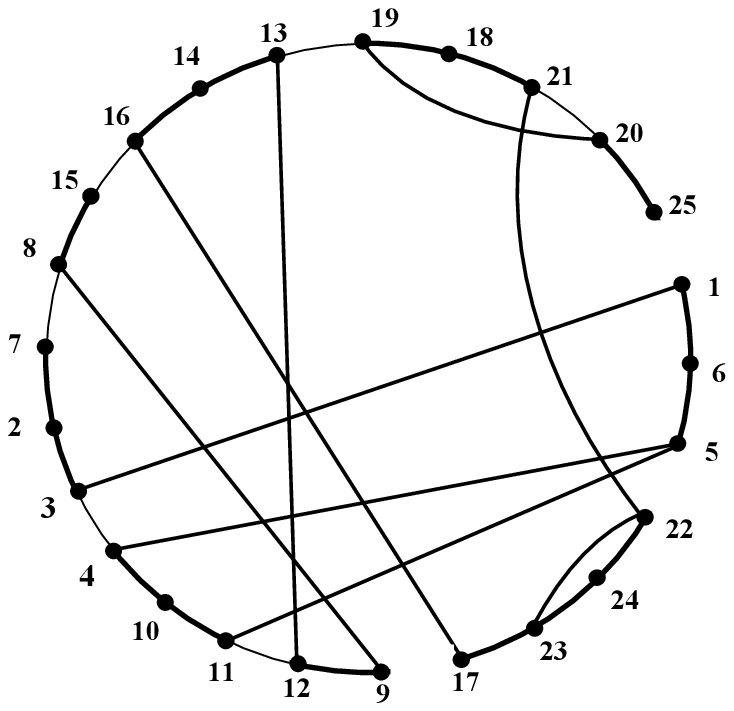}
\caption{A segment of the second object where an additional window
is formed by means of moving, between window nodes 22 and 25, the
segment $S_{1,5}$.\label{fig2-17}}
\end{figure}
\begin{figure}
\centering
\includegraphics[width=0.4\textwidth]{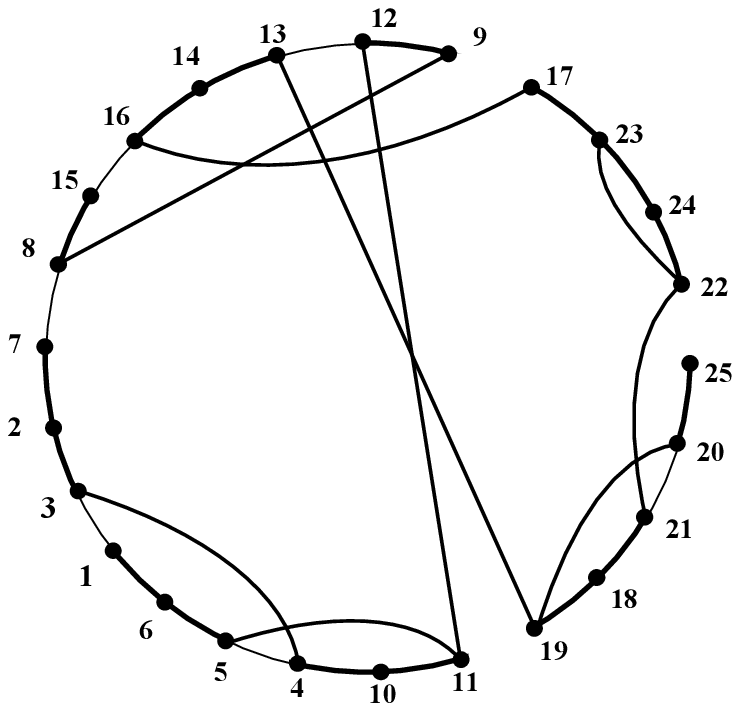}
\caption{A segment of the second object where an additional window
$w(11-19)$ is formed by moving edge $e(12-13)$ to the contour, and
edges $e(11-12)$, $e(19-13)$ to the interior.\label{fig2-18}}
\end{figure}
\begin{figure}
\centering
\includegraphics[width=0.4\textwidth]{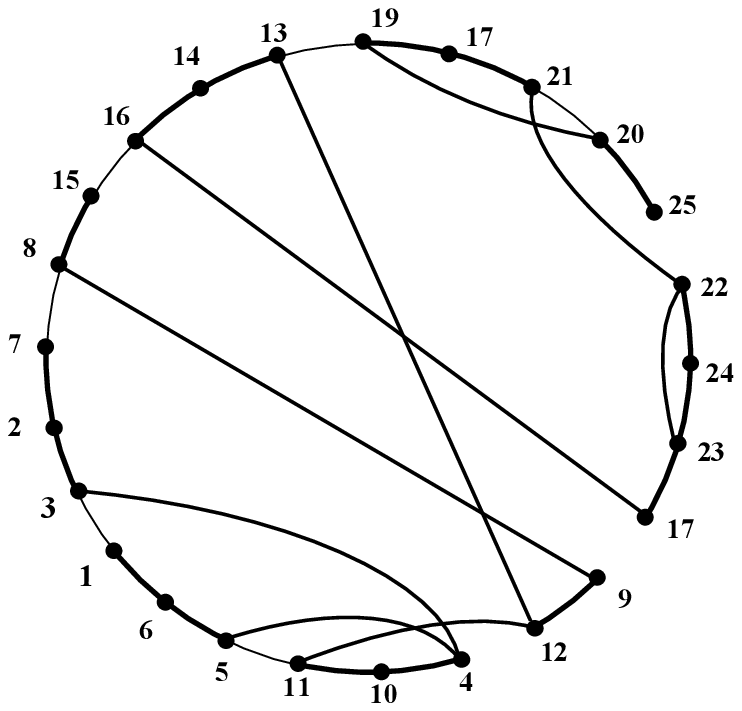}
\caption{A segment of the second object, where additional window
$w(4-12)$ is formed by moving edge $e(5-11)$ to the contour, and
edges $e(4-5)$, $e(11-12)$ to the interior.\label{fig2-19}}
\end{figure}
\begin{figure}
\centering
\includegraphics[width=0.4\textwidth]{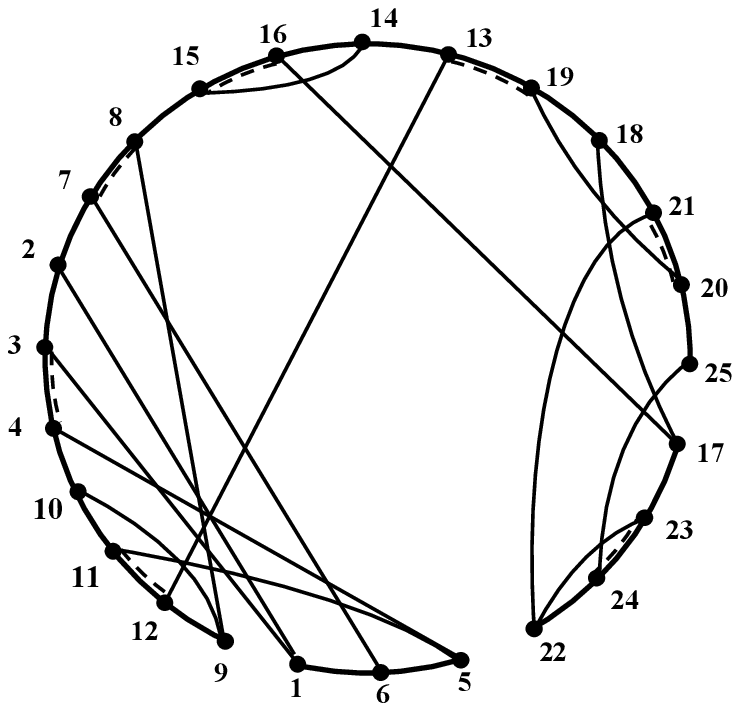}
\caption{The second basic object where temporarily deleted
interior edges are reinstated. Common contour edges are marked
with dashed lines.\label{fig2-20}}
\end{figure}

\begin{enumerate}
\item[(1)] There exists a free edge $e(5-11)$ and node 1 of degree
$d = 2$ that are adjacent to the free edge $e(3-4)$ on one side
(see Figure 2.16). An additional window is introduced by moving
part of the contour bounded by nodes $1$ and $5$ between any two
nodes that form a window. In this example, window $w(22-25)$ or
$w(9-17)$. We move this part of the contour between nodes 22 and
25 (see Figure 2.17). \item[(2)] Free edges $e(5-11)$ and
$e(19-20)$ are adjacent to free edge $e(12-13)$ on two different
sides (see Figure 2.16). An additional window $w(11-19)$ is
introduced by moving edge $e(12-13)$ to the contour, and edges
$e(11-12)$ and $e(19-13)$ to the interior (see Figure 2.18).
\item[(3)] Free edges $e(3-4)$ and $e(12-13)$ are adjacent to free
edge $e(5-11)$ on two different sides (see Figure 2.16). An
additional window $w(4-12)$ is introduced by moving edge $e(5-11)$
to the contour, and edges $e(4-5)$ and $e(11-12)$ to the interior
(see Figure 2.19). \item[(4)] Free edge $e(19-20)$ (see Figure
2.16) can not be used to introduce an additional window because
node 20 does not have adjacent nodes that satisfy the possibility
of introduction of additional windows (edge $e(21-22)$ is not a
free edge, and node 25 forms a window).\end{enumerate} From all of
these possibilities, only one can be implemented. Two or more
possibilities can not be implemented in this example because,
otherwise, the condition of absence of links between windows will
be violated. The second object with the temporarily deleted edges
being reintroduced is displayed (see Figure 2.20). Comparing the
first (see Figure 2.1) and the second basic objects (see Figure
2.20) we observe the following:
\begin{enumerate}
\item The sets of interior edges for each basic object are
disjoint. \item The contours for each basic object contain a
number of common edges: $e(3-4)$, $e(11-12)$, $e(7-8)$,
$e(15-16)$, $e(13-19)$, $e(21-20)$, $e(23-24)$, that appeared
during the construction of the second basic object. These edges
are marked by dotted lines (see Figure 2.20). \item No links exist
in either basic objects.\end{enumerate}\end{example}

\section{Selection of edges and windows}\label{sec3}

Suppose that we have constructed both basic objects for a graph of
degree $d \leq 3$. Suppose that such a graph contains one or more
Hamiltonian cycles. Edges of the graph can be naturally separated
in three disjoint groups:
\begin{enumerate}
\item[(1)] Edges that belong to all Hamiltonian cycles. \item[(2)]
Edges that belong to no Hamiltonian cycles. \item[(3)] Edges that
belong to some but not all Hamiltonian cycles.\end{enumerate}

Some edges from both basic objects can be separated into two
disjoint groups:
\begin{enumerate}
\item[(1)] Interior edges of both basic objects (these sets of
disjoint). \item[(2)] Common edges of the contours of both basic
objects.
\end{enumerate}

The procedure of selection of edges includes not only separation
into group, but also the assigning of weights to some edges that
determines the weights of all other edges, and of windows in both
basic objects. The assigning of weights can be performed for an
additional edge or window of the contour, or to a group of edges
and windows.

\subsection{Assigning of weights to the edges of both basic objects}
\subsubsection{Selection of contour edges and windows}

To any edge of the contour, a weight can be assigned. We assign a
particular weight (-1, 0, or +1) to an edge or window of the
contour of a basic object, that determines the weights of all
other edges and windows of the basic object. Initially, the
weights of all edges and windows are 0. Suppose that we select a
contour edge $e(m-n)$. If $N$ is odd, we assign weight -1 to this
edge, and then we perform the following operations:
\begin{enumerate}
\item Assign weight -0.5 to nodes $m$ and $n$. \item Moving along
the nodes of the contour, starting with node $s$ adjacent to node
$m$ we assign to each node alternating weights +0.5, -0.5. \item
The weight of every edge and every window is assigned by adding
the weights of their adjacent nodes. With this approach we obtain:
\begin{enumerate}\item The weight of one particular edge (edge $e(m-n)$)
becomes -1. \item The weight of all contour edges and windows
other than $e(m-n)$ remain 0. \item Any interior edge attains the
weight attains a value that can be -1, 0 or 1.\end{enumerate}
\end{enumerate}
We could also assign the initial weight of -1 to a window
$w(k-l)$, and perform the same operations as above.

If $N$ is even, we assign weight 0 to edge $e(m-n)$, and then we
perform the following operations:
\begin{enumerate}
\item Assign weights +0.5 to node $m$, and -0.5 to node $n$. \item
Moving along the nodes on the contour, starting with node $t$
adjacent to node $m$ we assign to each node alternative weights
-0.5, +0.5. \item The weight of every edge and every window is
assigned by adding the weights of their adjacent nodes. With this
approach we obtain:
\begin{enumerate}\item The weight of all contour edges and windows
(including $e(m-n)$) remain 0.\item Any interior edge attains the
weight attains a value that can be -1, 0 or 1.\end{enumerate}
\end{enumerate}

Consider the result of the selection of edges of a basic object,
assuming that the contour of this basic object is a fictitious
Hamiltonian cycle. A fictitious Hamiltonian cycle is a contour of
a basic object whose windows, initially, have the same 0 weight as
its edge. Suppose that a graph contains one or more Hamiltonian
cycles. Suppose that one of the contour edges is selected and
assigned a weight of -1. Independently of whether or not this edge
belongs to a Hamiltonian cycle, the combined length of any
Hamiltonian cycle will be equal to -1. If the selected edge is
given a weight of +1, then the length of every Hamiltonian cycle
will also be +1. This is because, on the contour, by our
assumption, a fictitious Hamiltonian cycle, where the windows
become fictitious edges with 0 weight, and the selection of any
edge or window implies the same change in the length of any real
Hamiltonian cycles as the same of any fictitious Hamiltonian
cycle. Consider the following cases:
\begin{enumerate}
\item[(1)] Suppose that a selected edge belongs to all Hamiltonian
cycles. Then, by the construction of one such Hamiltonian cycle,
part of the contour edges should be substituted by interior edges.
Since the combined weight of the substituted contour edges equals
0, and the total length of the Hamiltonian cycle is determined by
the weight of the selected edge, then the combined weight of the
interior edges that substitute the contour edges should also be 0.
\item[(2)] Suppose that the selected edge of weight -1 belongs to
a some but not all Hamiltonian cycles. Regardless of which
Hamiltonian cycle is constructed (with or without this edge), the
combined weight of the interior edges that substitute contour
edges in the Hamiltonian cycle will be equal to the combined
weights of those contour edges. If the selected edge belongs to
the constructed Hamiltonian cycle, then this combined weight will
be 0. If the selected edge does not belong to the constructed
Hamiltonian cycle, then the combined weight will be -1. \item[(3)]
Suppose that the selected edge of weight -1 does not belong to any
Hamiltonian cycle. Then, in the construction of a Hamiltonian
cycle, this edge should be substituted, and the combined weight of
the interior edges that are used in the Hamiltonian cycle, and the
combined weights of the contour edges that are not used in the
Hamiltonian cycle, will be -1. \item[(4)] Suppose that we select a
window of weight -1. It can be considered as a fictitious edge
that does not belong to any Hamiltonian cycle, and so should be
substituted by an interior edge. In this case, the previous
argument is valid: in the construction of a Hamiltonian cycle that
does not contain the selected window, the combined weights of the
interior edges that are used in the Hamiltonian cycle, and the
combined weights of the contour edges that are not used in the
Hamiltonian cycle, will be -1. \item[(5)] The claims (1)--(4) are
valid in the case of selection of any contour edge or window in
either basic object. However, one can select a set of contour
edges and/or windows. Since in the construction of a Hamiltonian
cycle, relative to any selected contour edge or window, the
combined weight of the substituting interior edges is equal to the
combined weight of the contour edges that are being substituted,
this property will be valid for any number of simultaneously
selected contour edges and windows.\end{enumerate}

Hence, the common argument is the following:
\begin{enumerate}\item In the construction of a Hamiltonian
cycle, the combined weight of substituted contour edges and
windows is equal to the combined weight of the substituting
interior edges, and this effect does not depend on whether
selected edges belong to the Hamiltonian cycle. \item If we
simultaneously select a set of contour edges and windows then the
combined weights of the substituting interior edges, and this
effect does not depend on whether selected edges belong to the
Hamiltonian cycle.\end{enumerate}

We note that the assigning of weights in turn assigns weights to
the edges of the basic objects. However, it is also convenient to
assign weights to the nodes of the basic objects. If a node is
assigned -0.5 or +0.5, then it is means that edges incident to
this node gain the weight -0.5 or +0.5.

\subsubsection{Selection of a group of edges}

We can select not just a single edge, but a group of edges.
Suppose that two basic objects of the same graph are constructed.
The contours of both basic objects contain common contour edges.
Each edge of this group could or could not belong to a particular
Hamiltonian cycle, if any exist. Now we solve the problem of
selection of common contour edges. By means of this subproblem, we
will define some characteristics of the basic objects that will
determine the Hamiltonicity of the graph.

The procedure of selection of common contour edges is as follows.
The task is to select common contour edges of basic objects under
the following conditions:
\begin{enumerate}
\item The weights of the nodes incident to common contour edges in
each basic object should be identical. \item The weights of the
common contour edges could be -1, 0 or 1. \item The interior edges
of both basic objects should have 0 weights. If this is
impossible, the correction procedure is performed (see Subsection
3.2.3) and then the combined weight of interior edges should be 0
for each object.\end{enumerate}
\begin{remark}We can formulate a dual to the subproblem of selecting
common contour edges of the basic objects. It differs from the
above problem by the condition that the contour edges that are not
common should all have zero weight.\end{remark}

We can request, in both primal and dual subproblems, that weights
of the nodes incident to common contour edges should be identical,
not only in both basic objects, but in both subproblems as well.
In this case, the difference in the solution of this subproblem
will be in assigning weights to other nodes that are not incident
to common contour edges. In the primal subproblem, weights of
those nodes should be assigned in such a way that all interior
edges have zero weights. In the dual subproblem, the contour edges
that are not common for both basic objects should have 0 weight.
These two can be simultaneously achieved. Furthermore, we show
that the solution of each of these subproblems can be used to
solve HCP.

\subsection{Assigning of weights to edges of the objects}
\subsubsection{Selection of common edges of contours of the
objects}

\begin{figure}
\centering
\includegraphics[width=0.85\textwidth]{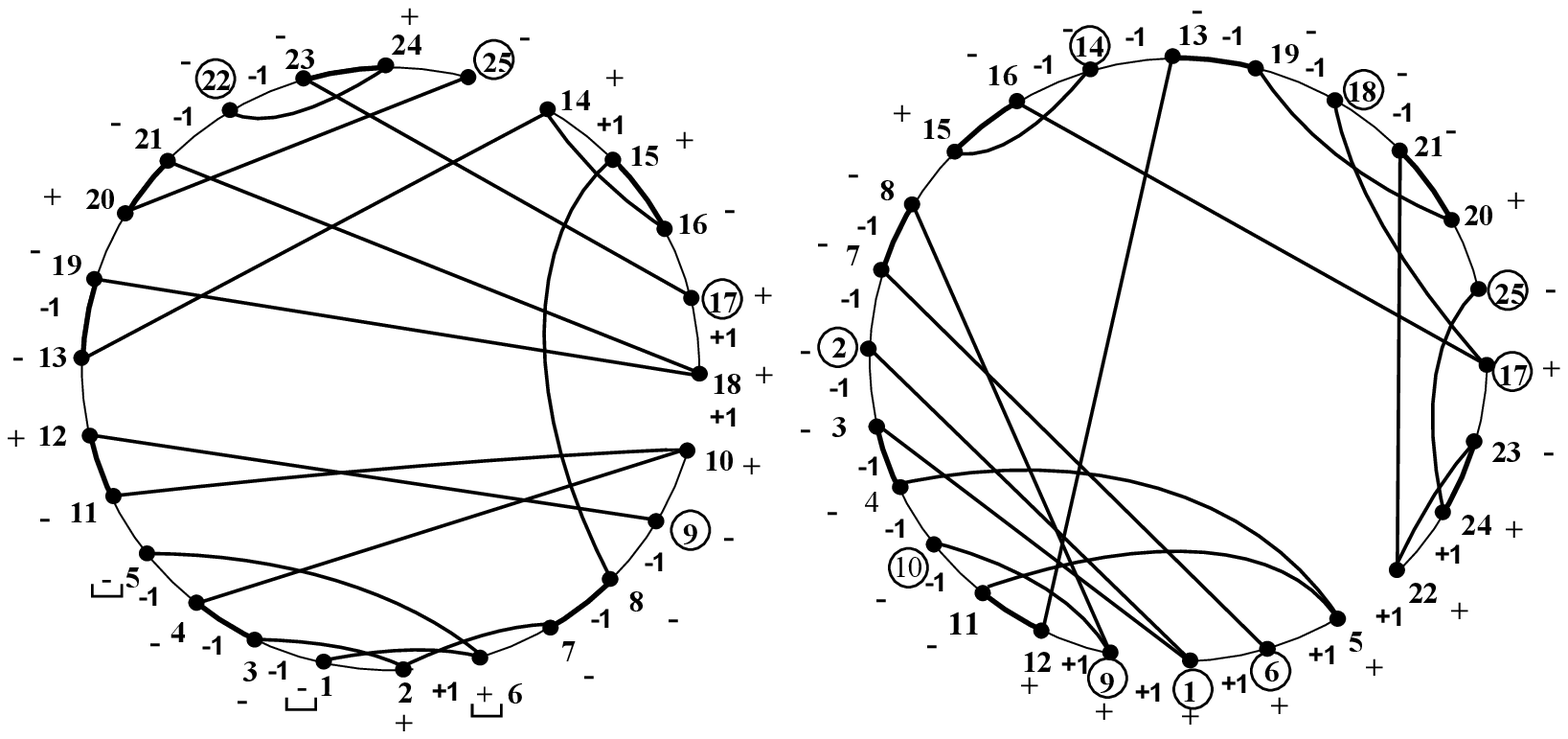}
\caption{Two basic objects, where the weights of the nodes are
displayed next to the node numbers. The nonzero weights of the
edges is displayed next to the edges.\label{fig3-1}}
\end{figure}
\begin{figure}
\centering
\includegraphics[width=0.4\textwidth]{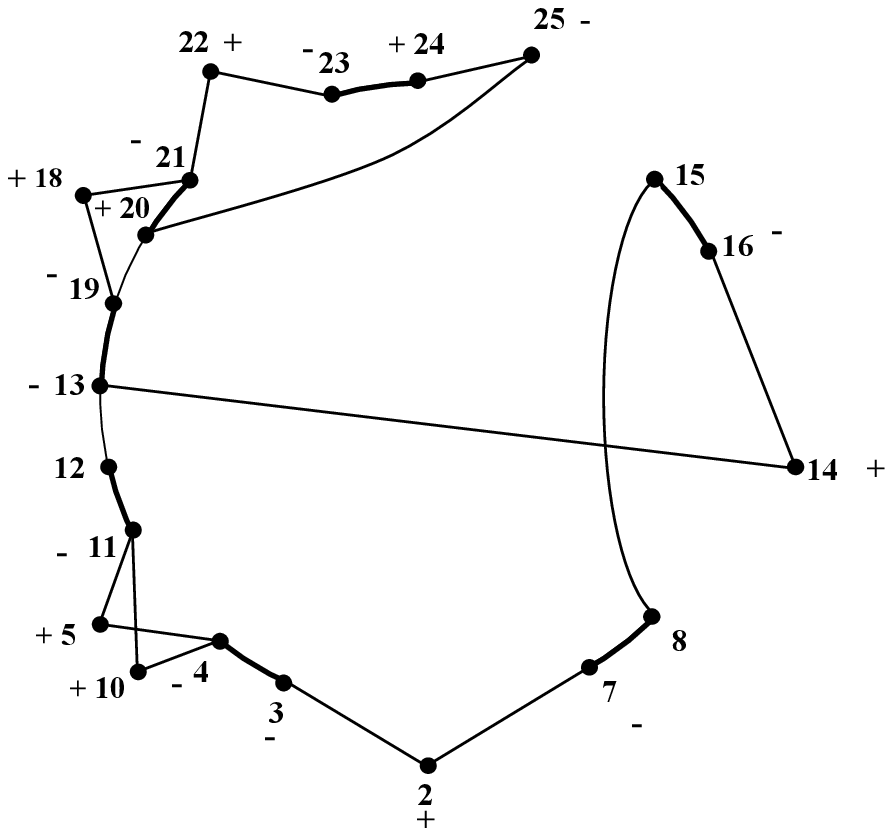}
\caption{Selected subgraph that determines the weight of the nodes
incident to common contour edges.\label{fig3-2}}
\end{figure}

Assigning of weights to edges of the objects is a part of the
solution to HCP. First, we will consider examples of assigning
weights to contour edges and interior edges that satisfy the
required conditions. Also, we will determine the sequence of such
assignment algorithm.

\begin{enumerate}\item[(1)] Determine the set of contour edges for each
object. The set of common contour edges is determined by comparing
the two sets of contour edges together. \item[(2)] One of the
conditions of assigning weights to the edges implies that the
nodes of the common contour edges should have identical weights in
both objects. For this reason, these nodes, and the contour edges
between them, should be selected as a separate subgraph common to
both objects.\end{enumerate}

On the basis of two objects, select a subgraph that contains:
\begin{enumerate}\item Common contour edges. \item Edges of each
object that link two common contour edges. In both objects, they
are the same. \item Interior edges that are incident to window
nodes. These sets will be disjoint for each object.
\end{enumerate}

Consider the following example of the selection of a subgraph for
the two objects displayed in Figure 3.1.
\begin{example}
The subgraph (see Figure 3.2) consists of the following edges:
\begin{enumerate}\item Edges $e(3-4)$, $e(11-12)$, $e(13-19)$,
$e(20-21)$, $e(23-24)$, $e(15-16)$, $e(7-8)$ are common contour
edges. \item Interior edges $e(12-13)$ and $e(19-20$ that link
common contour edges. \item Pairs of edges that in each object
incident are incident to the window nodes. In the first object
they are $e(18-19)$, $e(18-21)$; $e(10-11)$, $e(10-4)$; $e(2-3)$,
$e(2-7)$; $e(14-13)$, $e(14-16)$. In the second object they are
$e(20-21)$, $e(20-23)$; $e(5-11)$, $e(5-4)$. For the node $25$
having $d=2$ edges $e(25-20)$, $e(25-24)$.
\end{enumerate}
\end{example}

\begin{enumerate}\item[(3)] For the nodes of the selected
subgraph, we need to assign the weights. This algorithm has to
satisfy a number of conditions, as follows.
\begin{enumerate}\item[(3.1)] The nodes that are incident to common
contour edges should have equal weight in both objects. Because
the subgraph is common in both objects, then the weights assigned
to its edges will be the same in both objects. \item[(3.2)] The
weight of the common edges can be -1, 0 or +1. \item[(3.3)] The
edges of the subgraph that link nodes incident to common contour
edges should have 0 weight. \item[(3.4)] The pairs of edges that
are adjacent to a window node of degree 3, that in one object are
interior edges, should have 0 weights. In the other object, this
pair of edges will be contour edges and the weight of these edges
is not yet determined, and could be -1, 0 or +1 depending on other
weight assignments.\end{enumerate}\end{enumerate}
\begin{figure}
\centering
\includegraphics[width=0.2\textwidth]{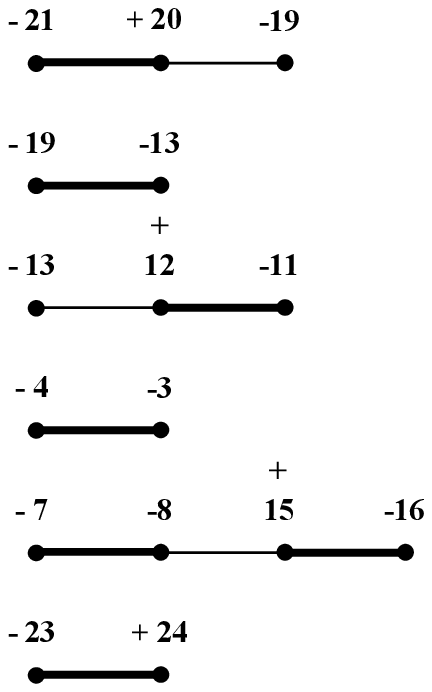}
\caption{Selected islands of the subgraph that determine the
weights of the nodes.\label{fig3-3}}
\end{figure}
\begin{figure}
\centering
\includegraphics[width=0.25\textwidth]{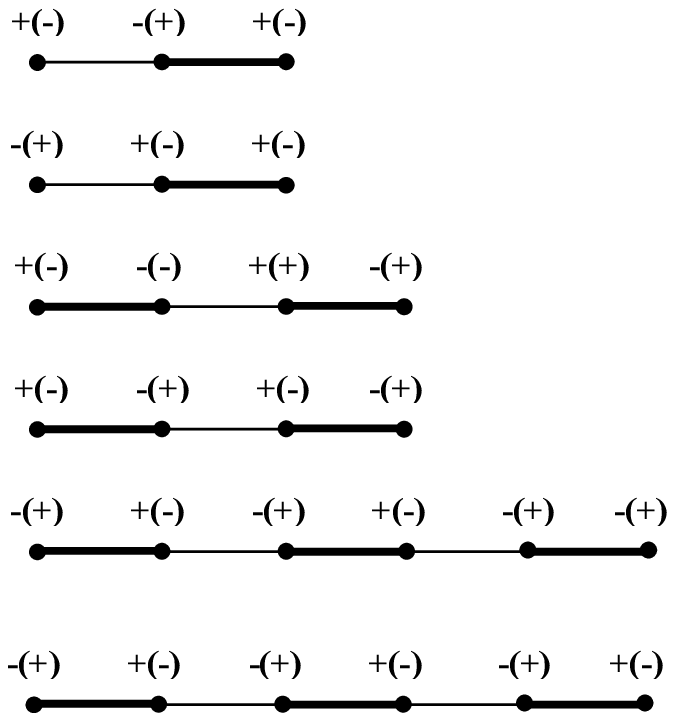}
\caption{The set of chains that determines the possibilities of
assigning weights to the nodes on the islands.\label{fig3-4}}
\end{figure}

\begin{example}Consider an example of weight assignment for the
subgraph shown in Figure 3.2. \begin{enumerate}\item[(A)] To all
window nodes, we assign the weight +0.5. In Figure 3.2 (and other
figures) we represent this assignment only by displaying the sign
of the weight. These nodes are nodes 2, 10, 18, 14, 22, 5 and 25.
In general, we can assign to these nodes the weight +0.5 or -0.5.
If there is a node in the subgraph linked to two window nodes,
then the two window nodes should be assigned the same weight,
either +0.5 or -0.5 simultaneously. In this example, the weights
for nodes 2, 14 and 25 can be assigned independently. The weights
of nodes 18 and 22, however, must be the same as each other. This
is because edges $e(18-21)$ and $e(22-21)$ are incident to node
21, that is incident to the common contour edge $e(20-21)$. This
is determined by the fact that two edges $e(18-21)$ and $e(22-21)$
are interior edges (one edge is interior in one object, and the
other edge is interior in the other object), and should therefore
have 0 weight. An analogous argument is valid for nodes 10 and 5,
for which the weight should therefore be equal. \item[(B)] Because
the weight of the edges incident to window nodes 2, 10, 18, 14,
22, 5 and 25 should be equal to 0, then nodes 23, 21, 19, 13, 11,
4, 3, 7, 16, 20, 24 have weight -0.5. Since edge $e(20-19)$ is
interior in one of the basic objects, and has weight -1, then we
need to reassign the weight of one of the windows nodes from +0.5
to -0.5. Nodes 19 and 20 that are incident to this node in both
object have interior edges that link them to window nodes (18 and
25). The weight of one of these nodes should be changed from +0.5
to -0.5. We assign the weight -0.5 to node 25 having $d=2$. Then,
nodes 20 and 24 should be assigned the weight
+0.5.\end{enumerate}\end{example}

\begin{enumerate}\item[(4)] In the subgraph, we select islands that are parts
of the subgraph such that their extreme (endpoint) nodes have
weights -0.5, or the weight is unassigned. An island can not
contain pairs of edges that are incident to window nodes. For this
example, the islands are displayed in Figure 3.3. In assigning
weights to nodes of an island, one can use patterns displayed in
the example of assigning weights to the nodes of chains, including
different number of common edges. If the subgraph contains chains
with common contour edges (see Figure 3.4), then the extreme nodes
$a$ and $b$ of the chains can have any weights associated with the
choice of the weight of common edges, in particular, -0.5. In
Figure 3.4, we display examples of assigning weights to the nodes
of chains containing one, two or three common edges. By variation
of the choice of weight for a particular edge, with any number of
common edges in the chain, the extreme nodes can be assigned equal
or opposite weights. We note that the situation when the extreme
nodes are assigned opposite weights remains unchanged for any
number of common edges in the chain. Furthermore, we need to
assign weights to nodes of the islands that can have only one
value. Such is node 12 whose weight can only be +0.5. We assign
weights to chain $7-8-15-16$ (see Figure 3.3). They are assigned
according to the rule displayed in Figure 3.4. If, to node 8, we
assign the weight +0.5, then to node 15 we should assign the
weight -0.5, and vice versa.\end{enumerate}

This completes the algorithm of selection of common contour edges,
and assigning weights to the corresponding nodes. As a result of
this algorithm, the nodes that are incident to common contour
edges are assigned identical weights in both objects. Also, the
edges of the subgraph that belong to three selected sets are
assigned admissible weights. In Figures 3.1, 3.2 and 3.3, the
signs of the weights are displayed next to the nodes.

\subsubsection{Assignment of weights to the nodes of the
objects}

As a result of the selection of common contour edges, some but not
all nodes of the objects are assigned weights. The above algorithm
allows us to assign weights to the nodes incident to common
contour edges, and also to the window nodes. The next problem is
to assign weights to the nodes that have not yet been assigned a
weight. The missing weights are assigned according to the
condition that interior edges must have 0 weights, and are
therefore determined by the weights assigned already to other
nodes in the above algorithm. In Figure 3.1, the nodes whose
weights are assigned at this step are circled. We note that
objects may contain nodes whose weights can not be assigned using
the above conditions. For example, in Figure 3.1, nodes 1, 6 and 5
are not assigned weights. This is because nodes 1 and 5 are not
incident to common contour edges. The choice of weight of a window
node (in this example, node 6) is made on the basis of comparison
of the sums of the weights of both objects, and is considered
later. The weight of any edge is defined as the sum of the weights
of the nodes incident to this edge. In Figure 3.1, the zero
weights are not displayed.

\subsubsection{Algorithm for assignment of weights to the edges of the objects}

\begin{enumerate}\item[(1)] Determine the contour edges for each
object. \item[(2)] On the basis of two objects, select a subgraph
that contains:
\begin{enumerate}\item Common contour edges. \item Edges of each
object that link two common contour edges. In both objects, they
are the same. \item The sets of interior edges that are incident
to window nodes. These sets of disjoint for each
object.\end{enumerate} \item[(3)] Assign weights to the nodes of
this subgraph using the algorithm outlined in Subsection 3.2.1.
\item[(4)] Identify the islands in the above subgraph and assign
weights to the nodes contained in the islands using the algorithm
also outlined in Subsection 3.2.1. \item[(5)] Assign weights to
remaining nodes using the algorithm outlined in Subsection 3.2.2.
\item[(6)] If in one or both objects there remains nodes without
weights assigned, then we assign to those nodes weights according
to the rules will be described later. \item[(7)] Determine the
weight of the edges as the sum of the weights of the nodes
incident to this edge. \item[(8)] If, among the interior edges of
one or both objects, there appears a nonzero weight, then we apply
the following correction algorithm for this object.\end{enumerate}

{\em Correction algorithm}
\begin{figure}
\centering
\includegraphics[width=0.85\textwidth]{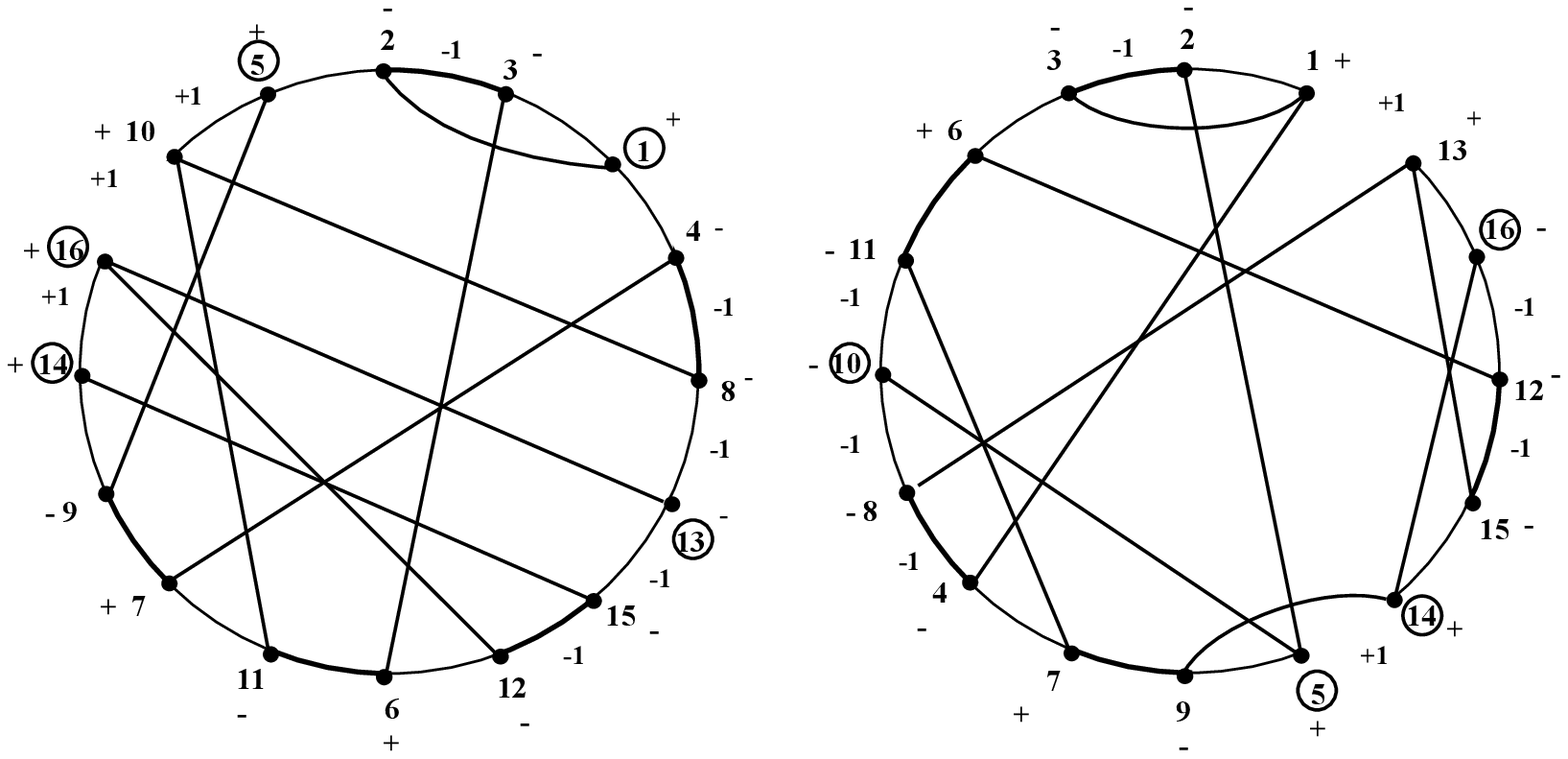}
\caption{Two basic objects, where the weights of the nodes are
displayed next to the node numbers. The nonzero weights of the
edges is displayed next to the edges.\label{fig3-5}}
\end{figure}

\begin{enumerate}\item[(1)] Determine the sum of the weights of
interior edges for both objects. Suppose that the sum is $-\gamma$
($+\gamma$). The number $\gamma$ can only be integer (-1, 0 or +1)
because the weight of each edge is determined by the sum of the
weights of its adjacent nodes. \item[(2)] Change the weight of the
edges incident to $\gamma$ nodes by +1 (-1) in such a way that the
sum of interior edges becomes 0. These edges cannot be
\begin{enumerate}\item Window nodes. \item Nodes incident to a
common contour edges. \item Nodes with degree $d =
2$.\end{enumerate}\end{enumerate}
\begin{figure}
\centering
\includegraphics[width=0.4\textwidth]{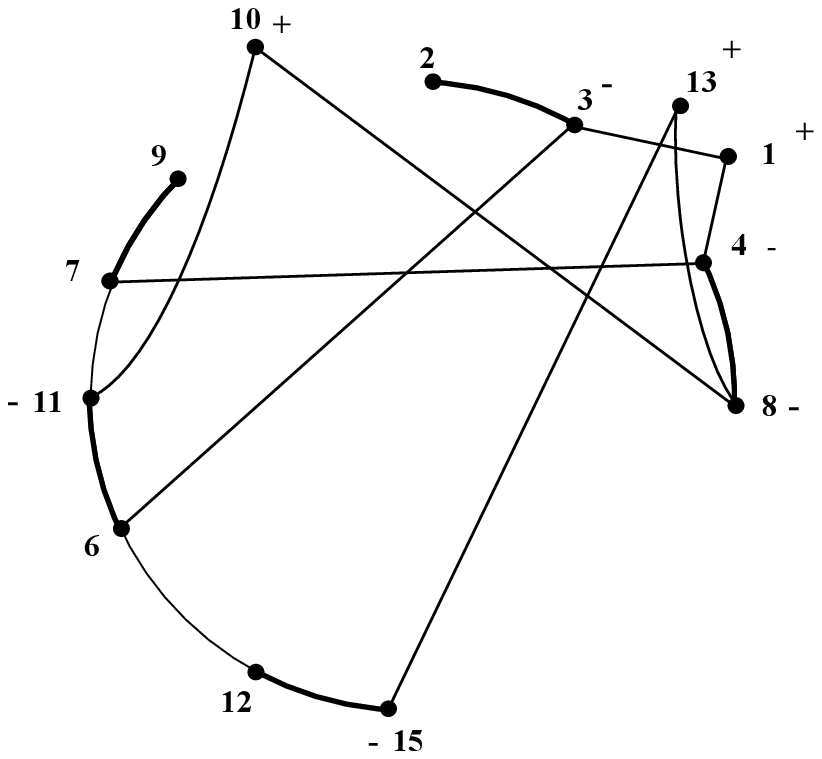}
\caption{Selected subgraph that determines the weight of the nodes
incident to common contour edges.\label{fig3-6}}
\end{figure}
\begin{figure}
\centering
\includegraphics[width=0.2\textwidth]{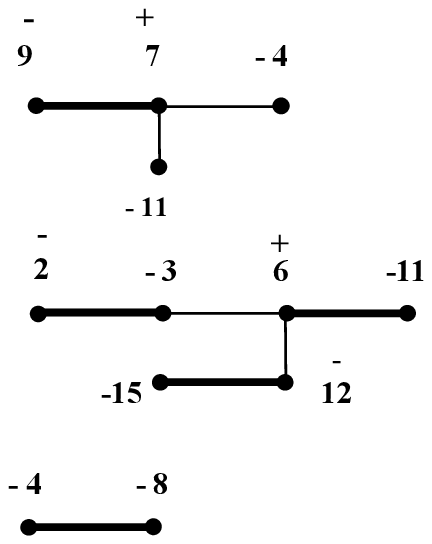}
\caption{Selected islands of the subgraph that determine the
weights of the nodes.\label{fig3-7}}
\end{figure}

At this stage, the general algorithm of assigning weights to the
edges of the objects is complete. Next, we consider examples of
assigning weights to edges of objects.
\begin{example}In Figure 3.5, we display two basic objects that
satisfy all required conditions. Now we explain how we assign
weights to the edges of these objects.
\begin{enumerate}\item[(1)] Determine the common contour edges. In
Figure 3.5, the common contour edges are displayed in bold, and
they are $e(2-3)$, $e(4-8)$, $e(12-15)$, $e(6-11)$, $e(7-9)$.
\item[(2)] On the basis of these two objects, select a subgraph
that contains:
\begin{enumerate}\item Common contour edges. \item Edges
$e(3-6)$, $e(4-7)$, $e(6-12)$, $e(7-11)$ that link nodes incident
to common contour edges. \item Interior edges incident to the
window nodes. These are $e(10-11)$ and $e(10-8)$ from the first
object, and $e(1-3)$, $e(1-4)$; $e(13-15)$, $e(13-18)$ from the
second object.\end{enumerate} \item[(3)] Assign weights to the
nodes of the selected subgraph.
\begin{enumerate}\item[(A)] To window nodes 10, 1, 13, assign
weight +0.5. If there is a node in the subgraph linked to two
window nodes, then these window nodes must be assigned the same
weight +0.5 (nodes 10 and 13). \item[(B)] To nodes 11, 3, 4, 8,
15, that are linked with window nodes by a single edge, assign the
weight -0.5.\end{enumerate}\item[(4)] In the subgraph, select the
islands. The pairs of edges incident to a window node can not be
part of any island. Assign weights to the nodes in these islands.
First, determine the nodes whose weights are predetermined by
previous weight assignments. In this example, nodes 7 and 6 can
only have weight +0.5. This choice is determined by the condition
that edges $e(7-11)$ and $e(3-6)$ must have 0 weight. Then, the
weight of node 12 can only be -0.5. The choice of weights for
nodes 9 and 2 can be taken arbitrarily. In our example, both edges
are assigned the weight -0.5. In Figures 3.5, 3.6 and 3.7, the
weights are displayed next to the nodes. \item[(5)] Assign weights
to the nodes that have not been assigned a weight in the previous
steps. They are assigned in accordance with the condition that the
weights of interior edges should be 0. The final assignment of
weights is displayed in Figure 3.5, where the weights of edges are
also displayed. The nodes for which the weights are determined
during this stage are circled.\end{enumerate}\end{example}
\begin{figure}
\centering
\includegraphics[width=0.85\textwidth]{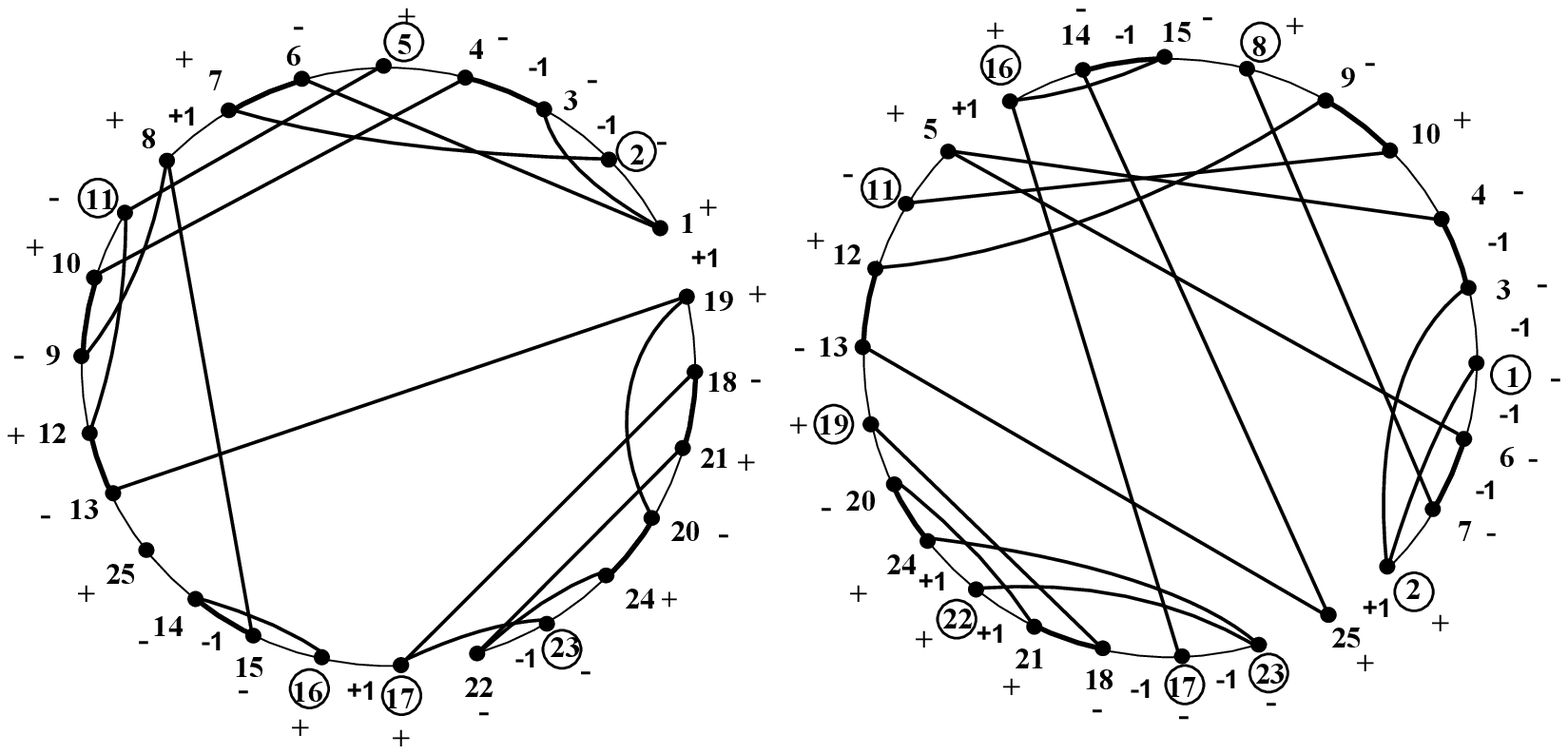}
\caption{Two basic objects, where the weights of the nodes are
displayed next to the node numbers. The nonzero weights of the
edges is displayed next to the edges.\label{fig3-8}}
\end{figure}
\begin{figure}
\centering
\includegraphics[width=0.4\textwidth]{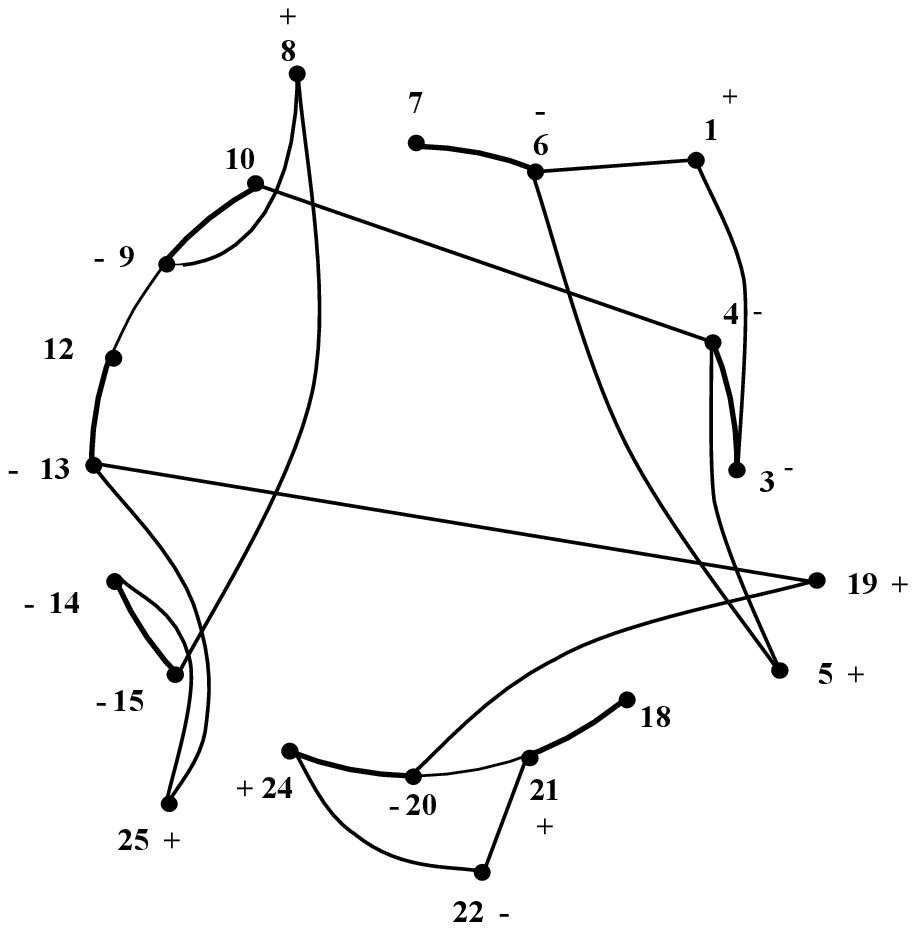}
\caption{Selected subgraph that determines the weight of the nodes
incident to common contour edges.\label{fig3-9}}
\end{figure}

\begin{example}In Figure 3.8, two basic objects satisfying the required conditions are
displayed. We assign weights to the edges of these objects.

\begin{enumerate}\item[(1)] Determine the common contour edges. In
Figure 3.8, the common contour edges are displayed in bold, and
they are $e(3-4)$, $e(6-7)$, $e(9-10)$, $e(12-13)$, $e(14-15)$,
$e(20-24)$, $e(18-21)$. \item[(2)] On the basis of these two
objects, select a subgraph (see Figure 3.9) that contains:
\begin{enumerate}\item Common contour edges. \item Edges
$e(9-12)$, $e(20-21)$, $e(4-10)$ that link nodes incident to
common contour edges. \item Interior edges incident to the window
nodes. These are $e(8-9)$ and $e(8-15)$; $e(1-3)$, $e(1-6)$;
$e(19-13)$, $e(19-20)$; $e(22-21)$, $e(22-24)$ from the first
object, and $e(25-13)$, $e(25-14)$; $e(5-4)$, $e(5-6)$ from the
second object.\end{enumerate} \item[(3)] Assign weights to the
nodes of the selected subgraph.
\begin{enumerate}\item[(A)] To window nodes 8, 1, 19, 5, 22, 25, assign
weight +0.5. \item[(B)] To nodes 13, 9 , 4, 6, 3, 21, 20, 24, 14,
15, that are linked with window nodes by a single edge, assign the
weight -0.5. Since edge $e(20-21)$ is interior in one of the basic
objects, and has weight -1, then we need to reassign the weight of
one of the windows nodes from +0.5 to -0.5. Nodes 20 and 21 that
are incident to this node in both object have interior edges that
link them to window nodes (19 and 22). The weight of one of these
nodes should be changed from +0.5 to -0.5. We assign the weight
-0.5 to node 22. Then, nodes 21 and 24 should be assigned the
weight +0.5.\end{enumerate}\item[(4)] In the subgraph, select the
islands. The pairs of edges incident to a window node can not be
part of any island. Assign weights to the nodes in these islands.
First, determine the nodes whose weights are predetermined by
previous weight assignments. In this example, nodes 10 and 12 can
only have weight +0.5. This choice is determined by the condition
that edges $e(9-12)$ and $e(4-10)$ must have 0 weight. Then, the
weight of nodes 7 and 18 can be taken arbitrarily. In our example,
both edges are assigned the weight -0.5. In Figures 3.8, 3.9 and
3.10, the weights are displayed next to the nodes. \item[(5)]
Assign weights to the nodes that have not been assigned a weight
in the previous steps. They are assigned in accordance with the
condition that the weights of interior edges should be 0. The
final assignment of weights is displayed in Figure 3.8, where the
weights of edges are also displayed. The nodes for which the
weights are determined during this stage are
circled.\end{enumerate}\end{example}
\begin{figure}
\centering
\includegraphics[width=0.2\textwidth]{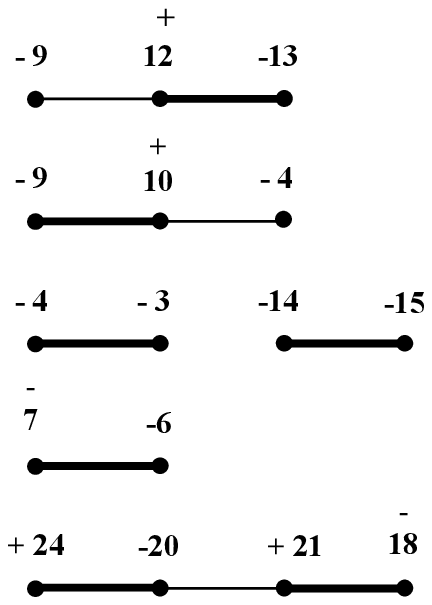}
\caption{Two basic objects, where the weights of the nodes are
displayed next to the node numbers. The nonzero weights of the
edges is displayed next to the edges.\label{fig3-10}}
\end{figure}
\begin{figure}
\centering
\includegraphics[width=0.85\textwidth]{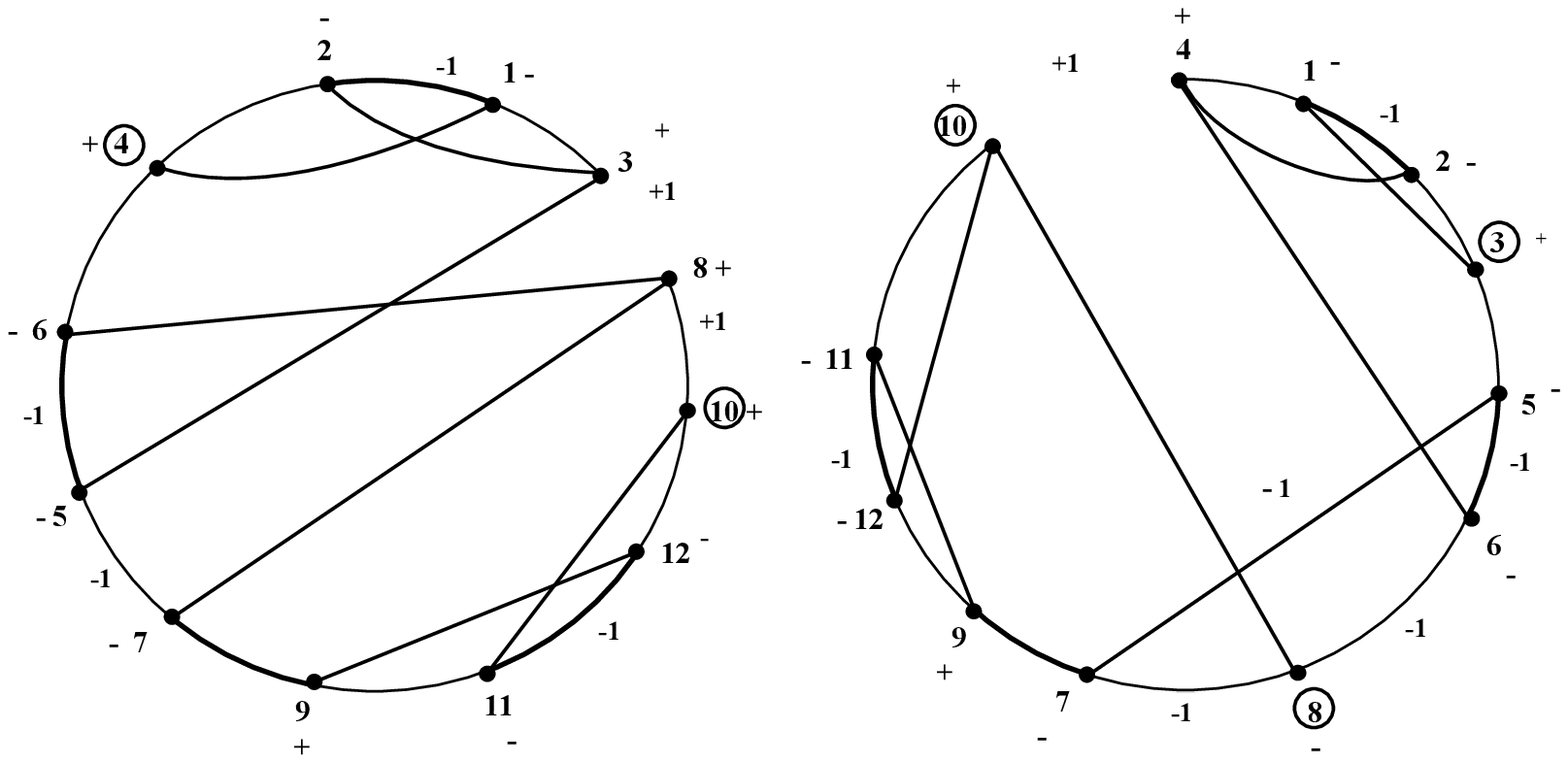}
\caption{Selected subgraph that determines the weight of the nodes
incident to common contour edges.\label{fig3-11}}
\end{figure}

\begin{example}In Figure 3.11, two basic objects satisfying the required conditions are
displayed. We assign weights to the edges of these objects.

\begin{enumerate}\item[(1)] Determine the common contour edges. In
Figure 3.11, the common contour edges are displayed in bold, and
they are $e(1-2)$, $e(5-6)$, $e(7-9)$, $e(11-12)$. \item[(2)] On
the basis of these two objects, select a subgraph (see Figure
3.12) that contains:
\begin{enumerate}\item Common contour edges. \item Edges
$e(5-7)$, $e(9-11)$, $e(9-12)$ that link nodes incident to common
contour edges. \item Interior edges incident to the window nodes.
These are $e(3-5)$ and $e(3-2)$; $e(8-6)$, $e(8-7)$; from the
first object, and $e(4-2)$, $e(4-6)$ from the second
object.\end{enumerate} \item[(3)] Assign weights to the nodes of
the selected subgraph.
\begin{enumerate}\item[(A)] In the subgraph there are nodes 2 and 6, that are
linked to two window nodes. Node 6 is linked to window nodes 4 and
8, and node 2 is linked to window nodes 3 and 4. In this case, all
window nodes 3, 4, 8 are assigned a weight +0.5. \item[(B)] Nodes
2, 6, 5, 7 that are linked to the window nodes, are assigned a
weight -0.5.\end{enumerate}\item[(4)] In the subgraph, select the
islands (see Figure 3.13). The pairs of edges incident to a window
node can not be part of any island. Assign weights to the nodes in
these islands. First, determine the nodes whose weights are
predetermined by previous weight assignments. Because one of the
islands contains a link between nodes 9 and 12, then these nodes
should be assigned opposite weights. Assign weight +0.5 to node 9,
and weight -0.5 to node 12. Then, the weight of node 11 must be
-0.5. In Figures 3.11, 3.12 and 3.13, the weights are displayed
next to the nodes. \item[(5)] Assign weights to the nodes that
have not been assigned a weight in the previous steps. They are
assigned in accordance with the condition that the weights of
interior edges should be 0. The final assignment of weights is
displayed in Figure 3.11, where the weights of edges are also
displayed. The nodes for which the weights are determined during
this stage are circled.\end{enumerate}\end{example}
\begin{figure}
\centering
\includegraphics[width=0.4\textwidth]{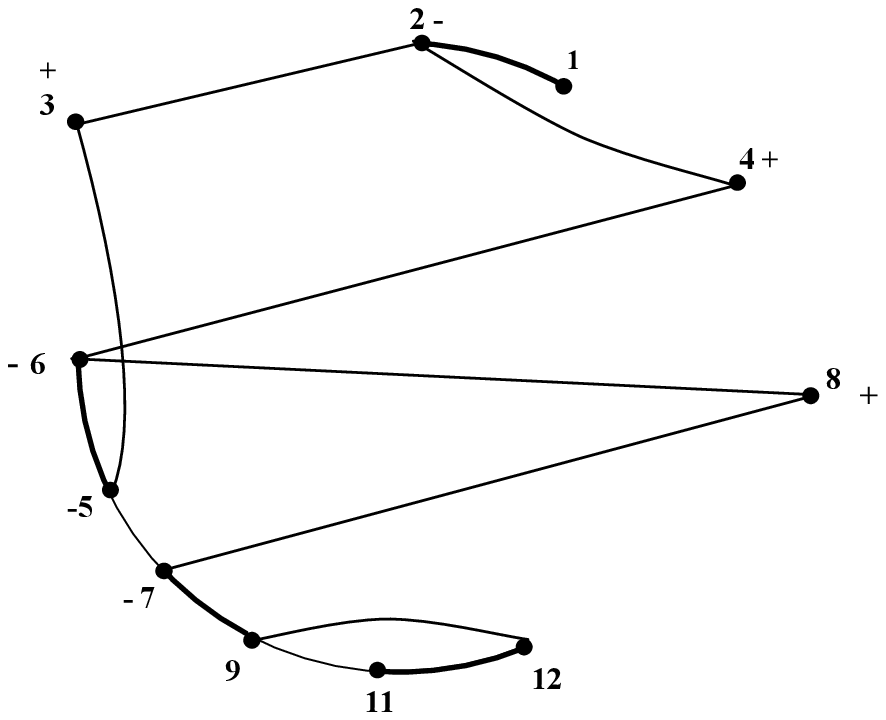}
\caption{Selected subgraph that determines the weight of the nodes
incident to common contour edges.\label{fig3-12}}
\end{figure}
\begin{figure}
\centering
\includegraphics[width=0.25\textwidth]{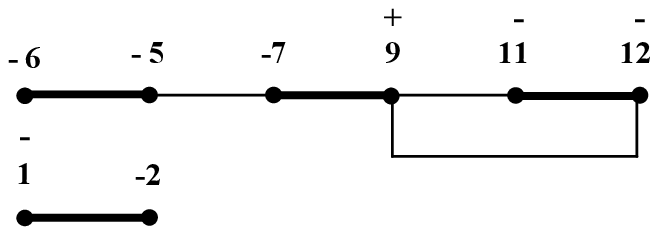}
\caption{Selected islands of the subgraph that determine the
weights of the nodes.\label{fig3-13}}
\end{figure}

\begin{remark}As follows from Figure 3.11, edge $e(5-7)$ that links
nodes incident to common contour edges $e(6-5)$ and $e(7-9)$ is
assigned a weight -1. In the second object, edge $e(5-7)$ is an
interior edge. This weight assignment contradicts the condition
that interior edges must have zero weight. After performing the
correction algorithm, which we now outline, we will still be able
to perform the test for Hamiltonicity of the graph. The underlying
principal of the correction algorithm will be that in the basic
objects, the sum of weights of interior edges must be 0. We need
to increase or decrease the weight of the edges incident to a node
of an object that contains an interior edge of nonzero weight.
This node can not be a window node (in this case, node 4 or 10),
nor nodes that are incident to common contour edges (in this case,
nodes 1, 2, 5, 6, 7, 9, 11, 12), as the weights of these node
should be the same in each object. Therefore, in this example, we
can only use nodes 3 and 8. We select node 3, and increase change
its weight from -0.5 to +0.5. This increases the weight of all
edges incident to node 3 by +1. This will ensure that the sum of
weights of the interior edges in the second object is zero. The
second object, with the corrected edge weights, is displayed in
Figure 3.14.
\begin{figure}
\centering
\includegraphics[width=0.4\textwidth]{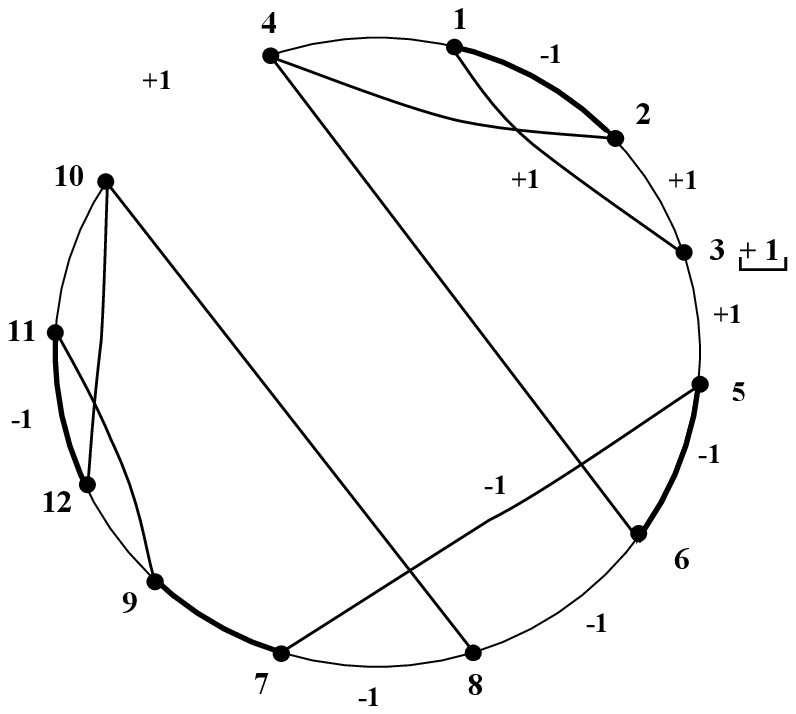}
\caption{The second object with the changed weights of edges after
the correction algorithm has been performed.\label{fig3-14}}
\end{figure}

Therefore, we can now determine the sum of weights of the contour
edges, and the sum of weights of the windows. These parameters are
necessary to determine the Hamiltonicity of the original
graph.\end{remark}

\section{Main theorem}\label{sec5}

\begin{thm}If both basic objects, after undergoing the algorithms
outlined in Section 3, have equal sums of weights of contour edges
($L_1 = L_2$), and equal sums of weights of their windows ($S_1 =
S_2$), then the original graph is Hamiltonian. If either of these
equalities are violated, the graph is not Hamiltonian.\end{thm}

\begin{proof}\begin{enumerate}\item[(1)] On transformed objects:

Consider basic objects that we subject to testing. Suppose that we
construct two basic objects. They do not intersect on interior
edges, or on uncommon contour edges. Suppose that the weight
assignment algorithm is completed.

We next take the union of both basic objects. The interpretation
of the union of one basic object with a second basic object is
that we receive the first basic object, where the weight for each
edge is the sum of weights of that edge in each basic object.
Thus, by finding the union of the first basic object with the
second, and also finding the union of the second basic object with
the first, we form two new basic objects that have the structure
of the two original basic objects, and equal weights. The interior
edges of the first union object will consist of the uncommon
contour edges of the second basic object, and vice versa. If, in
the first and second basic objects, the sum of contour weights was
the same ($L_1 = L_2$) and the sum of window weights was the same
($S_1 = S_2$), then in the union objects, the weight of the
contour will increase by the sum of weights of the common contour
edges, which is equal in both objects. The weight of the interior
edges for each union object will be equal to the sum of weights of
uncommon contour edges in the alternate basic object. If $L_1 =
L_2$, this implies that the sum of weights of interior edges in
both union objects will be the same. Thus, if in the first two
basic objects, $L_1 = L_2$, and $S_1 = S_2$, then likewise the sum
of weights of contour edges, and windows, will be the same in the
union objects as well. Using analogous arguments, it can be seen
that if $L_1 \neq L_2$, or $S_1 \neq S_2$, these inequalities are
preserved in the union objects as well.

We note that if we were required to perform the correction
algorithm, that led to $L_1 = L_2$ and $S_1 = S_2$, then these
equalities will be preserved for the union objects as well. This
is because the sum of the weights of interior edges is equal to 0,
and therefore will not change the sum of weights of the contour
edges in the union objects. If we now perform {\em equivalent
transformations}, which are described below, then the weight of
interior edges in the union objects will become 0 by means of the
transformations of the first kind listed below, that do not lead
to the violation of the equalities. If $L_1 \neq L_2$, or $S_1
\neq S_2$, then this inequalities will also be preserved in the
union objects after the correction algorithm is performed.

Suppose that in the first object the length of the Hamiltonian
cycle was $H_1$, and in the second it was $H_2$. Then in the union
objects, it will be equal to $H_1 + H_2$. The interior edges of
the union objects will have weights -1, 0 or +1.

\item[(2)] On equivalent transformations.

Equivalent transformations are performed with the aim to confirm
the equality of the length of Hamiltonian cycles in the
constructed basic objects. These lengths should be such that if
interior edges have zero weights, then the length of a Hamiltonian
cycle in each basic object should be equal to the length of its
contour. Since the interior edges in the constructed basic objects
have weights -1, 0 and +1, an equivalent transformation would
consist of consecutive annihilation of weights of -1 and +1 of
interior edges. These equivalent transformations can be divided
into two groups:
\begin{enumerate}\item The first group consists of the
transformations that, while annihilating the weight of an edge,
does not change the length of the contour, and therefore does not
change the length of a Hamiltonian cycle. \item The second group
consists of the transformations that do change the length of the
contour while annihilating the weight of an edge, and consequently
change the length of a Hamiltonian cycle by the same value
simultaneously in two basic objects.\end{enumerate}

Consider an equivalent transformation of the first type. Suppose
we can select $k$ interior edges of weight +1 and $k$ interior
edges of weight -1. We change the weight of $k$ interior edges of
weight +1 by using weight -1 that is assigned to nodes that are
incident to these interior edges, excluding window nodes.
Analogously, we annihilate the weight of $k$ interior edges of
weight -1 using weight +1 that is assigned to nodes that are
incident to these interior edges, excluding window nodes. This
allows us to annihilate the weight of $2k$ interior edges.

Consider a transformation of the second type. This transformation
is equivalent if, in both basic objects, there appears a weight of
the same sign. Suppose that $H_1 = H_2$ ($L_1 = L_2$, $S_1 =
S_2$). Then, the entire annihilation of the weights of interior
edges requires adding the same number of weights of the same sign
in each basic object that leads to equal, simultaneous, change in
the length of a Hamiltonian cycle, by assigning zero weights to
interior edges of both basic objects. As in the previous case, the
weights are added to nodes that do not form windows. The number of
added weights depends on the number of interior edges of one basic
object with positive weight ($r+$), and the number of interior
edges in the same basic object with negative weight ($r-$).

\begin{remark}It is possible to identify other equivalent
transformations of the second type, using window nodes, however
all of them change the length of contour, and therefore a
Hamiltonian cycle by the same value.\end{remark}

\item[(3)] On Hamiltonian graphs.

Suppose that in two union objects, the sum of contour weights are
equal ($L_1 = L_2$) and the sum of window weights are equal ($S_1
= S_2$), that was also true in the original basic objects. It is
known that if the weights of the interior edges are all zero,
then:
\begin{enumerate} \item In the construction of cycles, we
will omit some of the contour edges, and all of the window edges.
\item We substitute these edges with some interior edges. \item
The sum of the weights of contour edges and windows that are
omitted is equal to the sum of the weights of the interior edges
they are substituted with.\end{enumerate}

Since the sum of the weights of interior edges that we use for the
substitution equals zero (because all interior edges have zero
weight), then the length of any possible Hamiltonian cycle is
equal to the length of the contour ($H_1 = L_1 + S_1$, $H_2 = L_2
+ S_2$).

The remainder of the proof will show, by means of equivalent
transformations, that two union objects will take a form that will
allow us to determine the length of any Hamiltonian cycle in these
forms. This will be done by ensuring that the interior edges in
both union objects have zero weight.

We perform equivalent transformations in two union objects.
Because the sum of weights of the interior edges in both union
objects are equal to each other, then the composition of the set
of interior edges in these objects is the following:
\begin{enumerate}\item[(1)] Without loss of generality, assume that the
weight of the interior is nonnegative. In the first union object,
there are $\mu_1$ edges with negative weight, and at least $\mu_1$
edges with positive weight. After equivalent transformations of
type 1, we denote the number of remaining interior edges of
positive weight as $m \geq 0 $. \item[(2)] Using analogous
arguments to in (1), and the fact that the interior edges in both
union objects have equal weight, then after equivalent
transformations of type 1, the number of remaining interior edges
of positive weight is $m$ in the second union object as well.
\end{enumerate}

We can alter the weight of these $m$ edges in both union objects
using equivalent transformations of the second type, to reduce
their weight to zero. These equivalent transformations will change
the weight of the contour in both union objects by $-2m$, and
therefore the weights of the contours will remain equal to each
other. Thus, is in the basic objects $L_1 = L_2$ and $S_1 = S_2$,
then the equivalent transformations to the form where $H_1 = L_1 +
S_1$ and $H_2 = L_2 + S_2$, and therefore $H_1 = H_2$, which
confirms the existence of a Hamiltonian cycle.

\item[(4)] On non-Hamiltonian graphs.

Suppose that in two union objects, $L_1 \neq L_2$ and $S_1 \neq
S_2$. Without loss of generality, assume that the sum of weights
of the interior edges in the first union object is nonnegative.
Then, after we perform equivalent transformations of type 1, there
will remain $m \geq 0$ interior edges of positive weight in the
first union object. Then, in the second union object, there will
remain $k$ interior edges with nonzero weight, and the sum of
weights of the interior edges in each union object after these
equivalent transformations of type 1 is different.

At this stage, it is not possible to reduce the weights of all
interior edges in both union objects to zero by means of
equivalent transformations of type 2. Since the existence of a
Hamiltonian cycle implies that the length of any Hamiltonian cycle
should be equal in both union objects, this proves that graph is
non-Hamiltonian. This concludes the
proof.\end{enumerate}\end{proof}

Finally, we check the Hamiltonicity for the graphs displayed in
Figures 3.1, 3.5, 3.8 and 3.11.

In Figure 3.1, the two basic objects with the completed algorithm
of assignment of weights are displayed (see Example 4). It was
suggested that node $6$ in the first basic object (see Figure 3.1)
could be assigned weight -0.5 (or +0.5). The nodes 1 and 5 are
assigned weights +0.5 (or -0.5). Since we are trying to solve HCP,
which is determined by equating $H_1$ and $H_2$, the choice of
weights +0.5 or -0.5 is entirely determined by the condition $H_1
= H_2$. If node $6$ is assigned the weight +0.5 in the first basic
object, it implies that $H_1 = -5$ ($L_1 = -7$, $S_1 = 2$), and
for the second basic object $H_2 = -5$ ($L_2 = -7$, $S_2 = 2$).
Since $H_1 = H_2$, the graph, according to the main theorem, is
Hamiltonian.

In Figure 3.5, the two basic objects with the completed algorithm
of assignment of weights are displayed (see Example 5). As a
result of this algorithm, all edges and windows of the objects are
assigned weights -1, 0 or +1. We determine the sum of weights of
the contours of each object, which is supposed to be the length of
a Hamiltonian cycle. We obtain $H_1 = -2$ ($L_1 = -3$, $S_1 = 1$),
and $H_2 = -4$ ($L_2 = -6$, $S_2 = 2$). Because $H_1 \neq H_2$,
$(L_1 \neq L_2$, $S_1 \neq S_2$), then the graph, according to the
main theorem, is non-Hamiltonian.

In Figure 3.8, the two basic objects with the completed procedure
of assignment of weights are displayed (see Example 6). As a
result of this procedure, all edges and windows of the objects are
assigned weights -1, 0 or +1. We determine the sum of weights of
the contours of each object, which is supposed to be the length of
a Hamiltonian cycle. We obtain $H_1 = -1$ ($L_1 = -2$, $S_1 = 1$),
and $H_2 = -3$ ($L_2 = -5$, $S_2 = 2$). Because $H_1 \neq H_2$
($L_1 \neq L_2$, $S_1 \neq S_2$), then the graph, according to the
main theorem, is non-Hamiltonian.

In Figure 3.11, the two basic objects with the completed procedure
of assignment of weights, after the correction procedure is also
applied, are displayed (see Example 7). As a result of these
algoritms, all edges and windows of the objects are assigned
weights -1, 0 or +1. We determine the sum of weights of the
contours of each object, which is supposed to be the length of a
Hamiltonian cycle. We obtain $H_1 = -2$ ($L_1 = -3$, $S_1 = 1$),
and $H_2 = -2$ ($L_2 = -3$, $S_2 = 1$). Because $H_1 = H_2$ ($L_1
= L_2$, $S_1 = S_2$), then the graph, according to the main
theorem, is Hamiltonian.

\section{The general solution of the HCP with complexity estimate}

The solution of the HCP consists of several consecutive
subproblems:
\begin{enumerate}
\item[(1)] Construction of a basic object:
\begin{enumerate}
\item Construction of an object whose windows are not linked by
interior edges. \item Introduction of additional windows.
\end{enumerate}
\item[(2)] Construction of the second object:
\begin{enumerate}
\item Preliminary formation of the second object. \item
Construction of an object whose windows are not linked by interior
edges. \item Introduction of additional windows.
\end{enumerate}
\item[(3)] Selection of common edges of the contours of both
objects, including assignment of weights to the nodes of both
objects (and therefore to the edges of both object). \item[(4)]
Determination of parameters of both objects, and comparison of
these parameters to determine the existence of a Hamiltonian cycle
in the graph.\end{enumerate}

We will now estimate the complexity of each of the above
subproblems that are equivalent to the HCP.
\begin{enumerate}\item[(1)] Construction of a basic object:
\begin{enumerate}\item Construction of an object whose windows are not linked by
interior edges - this subproblem consists of separate individual
algorithms that eliminate links between windows. The order in
which these individual algorithms are performed is determined by
the configuration of the windows and the presence of degenerate
segments containing nodes whose degree $d = 3$. Each of these
procedures moves an interior edge linking window nodes to the
contour, and moves one of the contour edges to the interior. The
initial number of windows (some of which may violate the required
properties) is less than or equal to $N$. Each procedure can be
performed by a single search of all nodes, and no corrections (or
doubling back) are required. This proves that the complexity of
this subproblem is polynomial, because the number of algorithms
can not be more than $N$. \item Introduction of additional windows
-- this subproblem consists of separate individual algorithms that
introduce additional windows. The order in which these individual
algorithms are performed is determined by the configuration of
windows, and by the presence of free edges. Each of these
algorithms moves two of the contour edges to the interior, and one
free edge to the contour, creating an additional window. The
initial number of windows after the previous step is less than or
equal to $[N/6]$. During this algorithm, two contour edges are
substituted with one free edge, and with an additional window.
This substitution can be performed by a single search of all free
edges of the object, and no corrections (or doubling back) is
required. This proves that the complexity of this subproblem is
polynomial, because the number of algorithms can not be more
$[N/6]$.
\end{enumerate}
\item[(2)] Construction of the second object:
\begin{enumerate}
\item Preliminary formation of the second object -- this
subproblem consists of separate individual algorithms that
construct the initial formation of the second object. These
algorithms are the construction of the initial contour, the
determination of nodes that should not form windows, and temporary
elimination of these nodes. The order in which these individual
algorithms are performed is determined by the configuration of the
first basic object, i.e., the number its windows, the position of
the windows on the contour, the presence of nodes of the degree $d
= 2$, presence of degenerate segments, and the presence of free
edges. Once these algorithms have been completed, the final
algorithm is the temporary elimination of interior edges that
cannot belong to the contour of the second object, and for which
no further algorithms apply.
\begin{enumerate}
\item Construction of the initial contour -- this algorithm
consists of a sequence of rearrangements of the first basic object
by moving the interior edges to the contour, and rearranging some
nodes of degree $d = 2$. The number of such nodes is less than
$N$. \item Determination and elimination of nodes that should not
form windows -- this algorithm identifies nodes that are incident
to edges that are linked to window nodes of degree $d = 2$. The
algorithm of eliminating these nodes also involves the elimination
of links between windows. The number of nodes that can not form
windows of the second object is determined by the number of nodes
of degree $d = 2$, which is less than $N$. \item Elimination of
interior edges that cannot belong to the contour -- this algorithm
identifies interior edges and temporarily deletes them from the
object. The number of interior edges that are deleted is dependent
on the number of windows in the first basic object, and the number
of nodes with degree $d = 2$ in the second object.
\end{enumerate}

Hence, since each of the individual algorithms that construct the
preliminary form of the second object is performed less than $N$
times, each performed by a single search on the nodes of the
contour, and no corrections (or doubling back) are required, the
complexity of this subproblem is polynomial.

Elimination of interior edges that cannot belong to the contour
prohibits the movement of certain contour edges to the interior.
The subproblems 2(b) and 2(c) can be performed as described as in
1(a) and 1(b).
\end{enumerate}
\item[(3)] Selection of common edges and assignment of weights to
the edges of the object. This subproblem consists of individual
algorithms that:
\begin{enumerate}
\item Identify contour edges that are common for both basic
objects. \item Assigns equal weights, in both basic objects, to
nodes incident to common edges. \item Assigns zero weights to all
interior edges in either basic object. If this is impossible,
apply the correction algorithm to ensure that the sum of weights
of the interior edges is zero.

The algorithm of identifying common contour edges consists of
comparing contour edges of both objects, and is performed by a
single search of the contour edges of both objects.

The weights of the nodes that are incident to the common edges
should be the same. This is achieved through the assigning of
weights to the subgraph that is common to both objects. This will
ensure that the weight of every common contour edge will be the
same in both objects. The assignment of weights to the edges that
link nodes incident to common contour edges is achieved by means
of islands that are identified in the subgraph. In those cases
when the weight of an interior edge is nonzero, then we apply the
correction algorithm that makes the sum of the weight of interior
edges equal to zero. As a result of the weight assignment, the sum
of weights of the interior edges is zero. This is achieved by
assigning the weights +0.5 (or -0.5) to a node $a$, and if there
is an interior edge $e(a-b)$, then the weight -0.5 (or +0.5)
should be assigned to node $b$.

Therefore, since all algorithms of this subproblem are performed
no more than $N$ times, the complexity of this subproblem is
polynomial.
\end{enumerate}
\item[(4)] Determination of parameters of both objects -- the
parameters that determine the Hamiltonicity of the graph are sums
of weights of edges, and the sums of weights of windows of both
basic objects. If for both basic objects, these parameters are
equal, the graph is Hamiltonian. These parameters are determined
by adding the weights of less than $N$ contour edges, and the
weights of no more than $[N/6]$ windows, in each basic object.
Therefore, the complexity of computing these parameters is
polynomial.
\end{enumerate}

Therefore, the complexity of performing all of the above
subproblems is polynomial.

\section{Conclusion}

The described solution of the HCP based on the construction of two
basic objects for a given graph can be considered as surprising.
The HCP for graphs of degree $d \leq 3$ is NP-complete, but it is
just a particular case of HCP for general non-oriented graphs.
However, its solution provides a hope for the solution of the
general $P = NP$. The question of construction of basic objects
for other NP-complete problems remains open, and further hard work
is required to prove the existence of such objects.

It is worth mentioning that these results on the existence of
Hamiltonian cycles identify an algorithm to construct the cycle,
if it exists. A separate article will be devoted to the
construction of this algorithm.

\section*{Acknowledgments}

The author is grateful to V. Ejov (School of Mathematics and
Statistics, UniSA), and M. Haythorpe (School of Mathematics and
Statistics, UniSA) for useful discussions and questions.

\bibliographystyle{unsrtnumbered}


\end{document}